\documentclass{amsart}
\usepackage {amssymb}
\usepackage {amsmath}
\usepackage{amsthm}
\usepackage{graphicx}
\usepackage {amscd}
\usepackage {epic}
\usepackage[alphabetic]{amsrefs}
\usepackage[colorlinks, citecolor=blue]{hyperref}
\usepackage{comment}
\usepackage{tikz-cd}
\usetikzlibrary{cd}
\usepackage{soul,xcolor} 
\usepackage{enumitem}
\setstcolor{red}

\theoremstyle{plain}
\newtheorem{theorem}{Theorem}[section]
\newtheorem{corollary}[theorem]{Corollary}
\newtheorem{lemma}[theorem]{Lemma}

\newtheorem{remark}[theorem]{Remark}
\newtheorem{conjecture}[theorem]{Conjecture}

\newtheorem{proposition}[theorem]{Proposition}

\newcommand{\Pic}[0]{\operatorname{Pic}}

\newcommand{\fram}{\mathfrak{m}}

\newcommand{\Z}{{\mathbb{Z} }}

\newcommand{\mX}{{\mathcal{X}}}

\newcommand{\OO}{{\mathcal{O}}}
\newcommand{\Q}{{\mathbb{Q}}}
\newcommand{\XX}{\mathcal{X}}
\newcommand{\YY}{\mathcal{Y}}
\newcommand{\ZZ}{\mathcal{Z}}

\newcommand{\Supp}{{\rm Supp}}

\DeclareMathOperator{\Spec}{Spec}
\DeclareMathOperator{\Ann}{Ann}

\theoremstyle{definition}
\newtheorem{defn}[theorem]{Definition}

\newtheorem{defn-thm}[thm]{Definition--Theorem}  
\newtheorem{defn-prop}[thm]{Definition--Proposition}  
\newtheorem{defn-lem}[thm]{Definition--Lemma}  

\theoremstyle{remark}
\newtheorem{claim}[theorem]{Claim}

\def\commentbox#1{\textcolor{red}%
{\footnotesize\newline{\color{red}\fbox{\parbox{\textwidth}{\textbf{comment: } #1}}}\newline}}

\makeatletter
\newcommand*{\da@rightarrow}{\mathchar"0\hexnumber@\symAMSa 4B }
\newcommand*{\da@leftarrow}{\mathchar"0\hexnumber@\symAMSa 4C }
\newcommand*{\xdashrightarrow}[2][]{%
  \mathrel{%
    \mathpalette{\da@xarrow{#1}{#2}{}\da@rightarrow{\,}{}}{}%
  }%
}
\newcommand{\xdashleftarrow}[2][]{%
  \mathrel{%
    \mathpalette{\da@xarrow{#1}{#2}\da@leftarrow{}{}{\,}}{}%
  }%
}
\newcommand*{\da@xarrow}[7]{%
  \sbox0{$\ifx#7\scriptstyle\scriptscriptstyle\else\scriptstyle\fi#5#1#6\m@th$}%
  \sbox2{$\ifx#7\scriptstyle\scriptscriptstyle\else\scriptstyle\fi#5#2#6\m@th$}%
  \sbox4{$#7\dabar@\m@th$}%
  \dimen@=\wd0 %
  \ifdim\wd2 >\dimen@
    \dimen@=\wd2 %
  \fi
  \count@=2 %
  \def\da@bars{\dabar@\dabar@}%
  \@whiledim\count@\wd4<\dimen@\do{%
    \advance\count@\@ne
    \expandafter\def\expandafter\da@bars\expandafter{%
      \da@bars
      \dabar@ 
    }%
  }%
  \mathrel{#3}%
  \mathrel{%
    \mathop{\da@bars}\limits
    \ifx\\#1\\%
    \else
      _{\copy0}%
    \fi
    \ifx\\#2\\%
    \else
      ^{\copy2}%
    \fi
  }%
  \mathrel{#4}%
}
\makeatother

\begin{document}

\title[On the relative mmp for fourfolds]{On the relative Minimal Model Program for fourfolds in positive and mixed characteristic}
\author{Christopher Hacon} 
\address{Department of Mathematics \\  
University of Utah\\  
Salt Lake City, UT 84112, USA}
\email{hacon@math.utah.edu}
\author{Jakub Witaszek} 
\address{Department of Mathematics \\  
University of Michigan\\  
Ann Arbor, MI 48109, USA}
\email{jakubw@umich.edu}

\begin{abstract}
We show the validity of two special cases of the four-dimensional Minimal Model Program in characteristic $p>5$: for contractions to $\Q$-factorial fourfolds and in families over curves (``semi-stable mmp''). We also provide their mixed characteristic analogues. As a corollary, we show that liftability of positive characteristic threefolds is stable under the minimal model program, and that liftability of three-dimensional Calabi-Yau varieties is a birational invariant. Our results are partially contingent upon the existence of log resolutions.
\end{abstract}

\date{\today} 
\maketitle

\section{Introduction}
In recent years there has been much progress in understanding the birational geometry of threefolds over algebraically closed fields of characteristic $p>0$ by using methods inspired by the theory of F-singularities. One of the main achievements in this area is the proof of the minimal model program in dimension three (in the characteristic $p>3$ {case} and in the birational case, see \cite{hx13}, \cite{ctx13}, \cite{birkar13}, \cite{bw14}, \cite{HW19a}, \cite{HW19b}). The purpose of this paper is to show that the minimal model program also holds in some higher dimensional cases. This is the first step towards establishing the minimal model program in full generality for fourfolds in high characteristic. In view of the recent developments in mixed characteristic (\cite{BMPSTWW20}), we are also able to provide results on the minimal model program for arithmetic fourfolds.

In particular, we prove the following.
\begin{theorem} \label{t-mmp+} Let $(Y,\Delta)$ be a four-dimensional $\Q$-factorial dlt pair defined over a perfect field of characteristic $p>5$ or a DVR of characteristic $(0,p)$ for $p>5$ with a perfect residue field. Suppose that $\Delta$ has standard coefficients or that log resolutions of all log pairs with the underlying variety being birational to $Y$ exist (and are given by a sequence of blow-ups along the non-snc locus). Let $\pi \colon Y\to X$ be a projective birational morphism to a normal $\Q$-factorial variety $X$ such that ${\rm Ex}(\pi )\subset \lfloor \Delta \rfloor$. Then we may run a $(K_Y+\Delta)$-minimal model program over $X$,  which terminates with a minimal model.
\end{theorem}
\noindent The assumption on the residue field being perfect is necessary as the three-dimensional base point free theorem is not known in full generality in {a more general setting (for example when the residue field is $F$-finite)}.

As in \cite{HW19a}, this theorem allows us to derive a series of results for $\Q$-factorial four dimensional singularities, which in characteristic zero are proven using vanishing theorems. In particular, we show that { positive characteristic klt singularities are $W\OO_X$-rational, dlt modifications exist, and the inversion of adjunction holds} contingent upon the existence of a log resolution (Corollary \ref{c-witt-rational}, \ref{c-dlt-mod}, \ref{c-inversion-of-adjunction}).  Similarly, one should be able to extend other statements from dimension three in positive characteristic to dimension four, such as finite generatedness of the local Picard group (cf.\ \cite{CSK20}) or, assuming the BAB conjecture in dimension three, finiteness of the tame \'etale fundamental group. To avoid making our article too long, we do not write down proofs of these results here. \\

The second key result of this paper is the proof that the minimal model program can be run in families of threefolds over a curve. Note that when $\phi$ is smooth or $X$ is semi-stable, the $K_X$-mmp and the $(K_X+\Supp\,\phi^{-1}(s))$-mmp coincide. In particular, the following theorem proves the validity of the semi-stable minimal model program.

\begin{theorem} \label{t-mmp-families+} 
Let $(X,\Delta)$ be a four-dimensional $\Q$-factorial dlt pair with standard coefficients projective over a DVR $R$ with perfect residue field of characteristic $p>5$. Let $s \in \Spec R$ be its special point and let $\phi \colon X \to \Spec R$  be the natural projection. When $R$ is purely positive-characteristic, we also assume that it is a local ring of a curve $C$ defined over an algebraically closed field and that $(X,\Delta) := (\mX, \Phi) \times_C \Spec(R)$ for a $\Q$-factorial four-dimensional dlt pair  $(\mX, \Phi)$ which is projective over $C$.   

Further, suppose that log resolutions of all log pairs with the underlying variety being birational to $X$ exist (and are constructed by a sequence of blow-ups along the non-snc locus).

If $\kappa(K_X +\Delta\,/\,{\rm Spec}(R))\geq 0$ and ${\rm Supp}(\phi^{-1}(s))\subset  \lfloor \Delta \rfloor$, then we can run an arbitrary $(K_X+\Delta)$-minimal model program over $\Spec R$ which terminates with a minimal model. In particular, every sequence of steps of the mmp terminates. 
\end{theorem}
\noindent In mixed characteristic it is enough to assume that $K_X+\Delta$ is pseudo-effective, as then $\kappa(K_X+\Delta)\geq 0$ holds by the non-vanishing theorem for characteristic zero threefolds.

Theorem \ref{t-mmp-families+} can be used, for example, to construct canonical skeletons of degenerations of varieties with non-negative Kodaira dimension (\cite{NX16}), and allows for descending deformations of threefolds under operations of the minimal model program. In particular, we are able to provide an answer to \cite[Question 5.2]{HW17} for lifts to characteristic zero contingent upon the resolutions of singularities. This gives a new insight into \cite{LS14}.
\begin{theorem} \label{t-mmp-liftability} Let $(R,\mathfrak{m})$ be a DVR of mixed characteristic $(0,p)$ for $p>5$ and with a perfect residue field $k$. 

Let $X$ be a three-dimensional $\Q$-factorial terminal variety projective over $k$. Suppose that $K_X$ is pseudo-effective and that $X$ lifts to a scheme $\XX$ over $R$. Further, assume that log resolutions of all log pairs with the underlying variety being birational to $\XX$ exist (and are given by a sequence of blow-ups along the non-snc locus). Then the following hold.
\begin{enumerate}
    \item A (possibly non-$\Q$-factorial) minimal model of $X$ lifts over $R$.
    \item If $N^1(\XX/\Spec R)\to N^1(X)$ is surjective, then
    \begin{enumerate}
        \item every sequence of steps of a $K_X$-mmp lifts over $R$, and
        \item all $\Q$-factorial minimal models of $X$ lift over $R$.
    \end{enumerate} 
    \item If $K_X$ is big, then the canonical model of $X$ lifts over $R$.
\end{enumerate}
\end{theorem} 
\noindent Here, $N^1(X) = \Pic(X) \otimes \Q /{\equiv}$; this is the set of $\Q$-line bundles up to numerical equivalence. Further, we say that a scheme $X$ \emph{lifts} over $R$ if there exists a flat projective morphism $\XX \to \Spec R$ such that the central fibre $\XX_{\fram}$ is isomorphic to $X$. 

As a corollary to (2b), we get that liftability of three-dimensional terminal Calabi-Yau varieties is a birational invariant. Here we say that a projective variety $X$ is \emph{Calabi-Yau} if $\omega_X$ is trivial and $H^i(X,\OO_X)=0$ for $0 < i < \mathrm{dim} X$.
\begin{theorem} \label{t-calabi-yau-liftability}Let $(R,\mathfrak{m})$ be a complete DVR of mixed characteristic $(0,p)$ for $p>5$ and with a perfect residue field $k$. 

 Let $X$ and $Y$ be three-dimensional terminal $\Q$-factorial projective Calabi-Yau varieties defined over $k$. Suppose that $X$ and $Y$ are birational and that $X$ lifts to a scheme $\mathcal{X}$ over $R$. Further, assume that log resolutions of all log pairs with the underlying variety being birational to $\mathcal{X}$ exist (and are given by a sequence of blow-ups along the non-snc locus). 
 
 Then $Y$ lifts over $R$ as well.
\end{theorem}
\noindent It is necessary to assume that $R$ is complete to guarantee that $N^1(\XX/\Spec R) \to N^1(X)$ is surjective (see Lemma \ref{l-lblift}).\\

It is natural to wonder if the assumption on the existence of log resolutions can be dropped. In general, this seems to be a very difficult problem in view of the fact that alterations or quasi-resolutions do not behave nicely with respect to the canonical divisor. However, one may hope to remove this assumption if we start to run the mmp on a regular variety, and this is, to some extent, the case (see Proposition \ref{p-resolutions}) provided a much weaker assertion: that embedded log resolutions of subschemes of smooth fourfolds exist.  Note however, that the current proofs of Theorem \ref{t-mmp+} and of Theorem \ref{t-mmp-families+} (and so that of Theorem \ref{t-mmp-liftability} and Theorem \ref{t-calabi-yau-liftability} as well) require additionally that the log resolutions are constructed by a sequence of blow-ups along the non-snc locus (see the second paragraph of the proof of Theorem \ref{t-mmp+} and of the proof of Claim \ref{c-central}).

\subsection*{The idea of the proof of Theorem \ref{t-mmp+}}
Here, we explain the proof in positive characteristic. The mixed characteristic case is analogous with $S^0$ and $F$-regularity replaced by $B^0$ and $+$-regularity, respectively, as developed in \cite{BMPSTWW20}.

By the same arguments as in \cite{birkar13}, we can assume that $\Delta$ has standard coefficients. Our strategy follows the ideas developed in \cite{HW19a}. The main new difficulty is to show the existence of pl-flips occuring in the above mmp. As observed in \cite{HW19a}, the corresponding pl-flipping contractions contain a relatively ample divisor in the boundary. In particular, by applying inversion of F-adjunction twice and using the F-regularity of relative log Fano surfaces in the birational case (see \cite[Theorem 3.1]{hx13}) we can deduce that these pl-flipping contractions are relatively purely F-regular (up to a small perturbation). 

Unluckily, in contrast to the three-dimensional case (\cite{hx13}), this is not enough to conclude the existence of flips. Indeed, \cite{hx13} employs a strategy of Shokurov for constructing flips which is not valid in higher dimensions. On the other hand, it is not clear how to generalise the higher dimensional proof of the existence of flips (\cite{HMK10}) from characteristic zero to positive characteristic in this setting, as the proof calls for applying vanishing theorems on a log resolution, in which case the relative pure F-regularity is lost.

We address this problem by finding a new way of constructing flips when the boundary contains a relatively ample divisor.  
\begin{theorem}[{cf.\ Theorem \ref{t-sflip}}] \label{t-flips_exist+} Let $(X,S+A+B)$ be an $n$-dimensional $\Q$-factorial dlt pair defined over an $F$-finite field of characteristic $p>0$ and let $f \colon X \to Z$ be a $(K_X+S+A+B)$-flipping contraction with $\rho(X/Z)=1$ where $S$ is an $f$-anti-ample divisor, $A$ is an $f$-ample divisor, and $\lfloor B \rfloor = 0$. Suppose that the minimal model program is valid in dimension $n-1$. If $(X,S+(1-\epsilon)A+B)$ is relatively purely F-regular for all $0 < \epsilon < 1$, then the flip of $f$ exists.
\end{theorem}
In fact, we show that $H^0(X, m(K_X+S+A+B))\to H^0(S, m(K_{S}+A_{S}+B_{S}))$ is surjective for all $m>0$ sufficiently divisible where $K_S+A_S+B_S = (K_X+S+A+B)|_S$ (Proposition \ref{p-1}). In particular, the finite generatedness of canonical rings in dimension $n-1$ implies that the canonical ring of $(X,S+A+B)$ is finitely generated as well, and so its projectivisation is the sought-for flip. Note that this surjectivity is false without the ample divisor in the boundary (see Remark \ref{r-xu}).  

\subsection*{Mixed charactieristic}
The original version of this article covered the positive characteristic case only. A few months after its submission to the arxiv, \cite{BMPSTWW20} developed the mixed characteristic analogues of $S^0$ and $F$-regularity, and proved the validity of the three-dimensional minimal model program for arithmetic threefolds. In view of this and recent results on relative semiampleness (\cite{witaszek2021relative}), the original results of our article generalise to mixed characteristic with small modifications. The current updated version of our paper incorporates these generalisations and several applications thereof. 

We also note that results on the three-dimensional mmp in mixed characteristic in the semistable case were obtained in \cite{TY20} independently of \cite{BMPSTWW20}. The geometric part of their proof is based on the ideas of the original version of our article and on \cite{HW19a}. In particular, \cite{TY20} generalised our Theorem \ref{t-flips_exist+} to mixed characteristic before this updated version of our article was submitted.

The notions of $+$-regularity and $+$-stable sections ($B^0$) work in both positive and mixed characteristic, and so we could restate the whole article to avoid the use of $F$-regularity and Frobenius stable sections ($S^0$). We decided against doing so, instead treating positive and mixed characteristic cases separately, to prevent our article from being unnecessarily technical for the readers interested in the positive characteristic case only. The notion of $F$-regularity has been in use for a long time, in contrast to $+$-regularity which was introduced recently and is more technically involved. 

In the mixed characteristic case, some of our results are valid beyond the case of schemes defined over DVRs (see e.g.\ the setting of \cite{BMPSTWW20}). For the sake of readability and in order to avoid dealing with unnecessary technicalities, we do not strive to provide the most general version. Note however that there is a fundamental obstacle to generalising Theorem \ref{t-mmp-families+} to the case of Dedekind domains such as $\Spec \mathbb{Z}$ (or $\Spec k[t])$) due to the issues with termination of flips; indeed, over such a base it could \emph{a priori} happen that there is an infinite sequence of flipping curves contained in fibres over different prime numbers. Furthermore, Theorem \ref{t-mmp-families+} needs the residue field to be perfect so that we can invoke the three-dimensional base point free theorem as well as \cite{NT20}, and the positive characteristic case thereof requires essentially that the base spreads out  over an algebraically closed field so that we can apply positive characteristic Bertini theorems from \cite{SZ13}. 


To align with the notation in the positive characteristic case, we use the word \emph{variety} (as defined in the preliminaries below) even for schemes over DVRs. We hope that this will not cause any confusion for the reader.

\section{Preliminaries}
A scheme $X$ will be called a \emph{variety} if it is integral, separated, and of finite type over a field $k$ or a divisorial valuation ring $R$. Throughout this paper, unless otherwise stated, we work over $F$-finite fields $k$ of characteristic $p>0$ or DVRs $R$ of characteristic $(0,p)$ for $p\geq  0$.

We refer the reader to \cite{KM98} for the standard definitions and results of the minimal model program, to \cite{HW19a} for a brief introduction to F-singularities, and to \cite{BMPSTWW20} for the results on $+$-regularity. Further, we refer to \cite[Remark 2.7]{GNT16} for a discussion on  properties of varieties which are independent of a base change from a perfect field. We warn the reader that $\Q$-factoriality is not such a property.

In this paper, a {\it pair} $(X,B)$ consists of a normal variety $X$ and an effective $\Q$-divisor $B$ such that $K_X+B$ is $\Q$-Cartier. The pair $(X,B)$ is {\it kawamata log terminal} (\emph{klt}) (resp.\ {\it log canonical} (\emph{lc})) if for any proper birational morphism $f \colon Y\to X$ and any prime divisor $E$ on $Y$ we have ${\rm mult }_E(B_Y)<1$ (resp.\ ${\rm mult }_E(B_Y)\leq 1$) where $K_Y+B_Y=f^*(K_X+B)$.
If $(X,B)$ admits a log resolution $f \colon Y\to X$, then it suffices to check the above condition for all prime divisors $E$ on $Y$. 

The definitions of singularities in birational geometry may be found in \cite{kollar13} and \cite[Section 2.5]{BMPSTWW20}. We say that a pair $(X,B)$ such that $B=\sum b_iB_i$ with $0\leq b_i\leq 1$ is {\it divisorially log terminal} (\emph{dlt}) if there exists an open subset $U\subset X$ such that $U$ is regular and $B|_U$ has simple normal crossings support and for every proper birational morphism $f \colon Y\to X$ and any prime divisor $E$ on $Y$ with centre $Z$ contained in $X\setminus U$, we have ${\rm mult }_E(B_Y)<1$.
We say that $a_E(X,B):=1-{\rm mult }_E(B_Y)$ is {\it the log discrepancy} of $(X,B)$ along $E$.
A pair $(X,S+B)$ with $\lfloor S+B\rfloor =S$ irreducible, is {\it purely log terminal} ({\it plt}) if $a_E(X,S+B)>0$ for any $E\ne S$.

A morphism of schemes $f\colon X\to Y$ is a {\it universal homeomorphism} if for any morphism $Y'\to Y$, the induced morphism $X'=X\times _Y Y'\to Y'$ is a homeomorphism. In positive characteristic, it is known that a finite morphism is a universal homeomorphism if and only if it factors a sufficiently high power of the Frobenius morphism. We say that a variety $X$ is {\it normal up to universal homeomorphism} if its normalisation $X^\nu \to X$ is a universal homeomorphism. {We learned the following result from J\'anos Koll\'ar.
\begin{lemma}\label{c-nuuh} Let $(X,D+\Delta)$ be a dlt pair with $D$ prime and $\Q$-Cartier. Then $D$ is normal up to universal homeomorphism.
\end{lemma}
\begin{proof}
By \cite[Theorem 41]{kollar16}, it is enough to show that $D^{\mathrm sh} \,\backslash\, \{x^{\mathrm sh}\}$ is connected for the Henselisation $D^{\mathrm sh}$ of $D$ at every point $x \in D$. If $x$ is a point of codimension at most two in $X$, then this follows by surface theory (cf.\ \cite{kollar13}) as $(X,D+\Delta)$ is dlt. Thus, we may assume that $x$ is of codimension at least three in $X$. Let $X^{\wedge}$ and $D^{\wedge}$ be the completions of $X$ and $D$ at $x$, respectively. It is enough to show that $D^{\wedge} \,\backslash\, \{x\}$ is connected. Note that $X^{\wedge}$ is normal (\cite[Tag 033C, 07GC, 037D]{stacks-project} and \cite[Lemma 2.7]{popescu00}), hence $S_2$, and so $X^{\wedge} \,\backslash\, \{x\}$ is connected. Moreover, $X^{\wedge} \,\backslash\, \{x\}$ is S2 at every closed point $x' \in X^{\wedge} \,\backslash\, \{x\}$ (as the scheme is Catenary, the codimension of $x'$ in $X^{\wedge}$ is at least two). Since $D^{\wedge}$ is of codimension one (\cite[Tag 07NV]{stacks-project}) and $\Q$-Cartier, we can apply \cite[XIII.2, Theorem 2.1, p.\ 139]{Grothendieck62} to a Cartier multiple of $D^{\wedge}$ to conclude the proof. 
\end{proof}}

We refer to \cite{GNT16} and \cite[Subsection 2.3]{HW19a} for the notion of $W\OO$-rational singularities in positive characteristic. Since log resolutions of singularities are not known to exist beyond dimension three, a positive characteristic singularity $X$ is called \emph{$W\OO$-rational} if $R^if_*W\OO_{V,\Q}= 0$ for $i>0$ and every quasi-resolution $f \colon V \to X$ (see \cite[Section 3]{GNT16}). Note that it is enough to verify this condition on a single quasi-resolution (\cite[Corollary 4.5.1]{CR12}). In particular, if $X$ admits a usual resolution of singularities, we can verify this condition on it.

\begin{lemma}\label{l-dltad} Let $(X,\Delta)$ be a $\Q$-factorial dlt pair and $S$ an irreducible component of $\lfloor \Delta \rfloor$. If $S^\nu \to S$ is the normalisation and $K_{S^\nu}+\Delta _{S^\nu}=(K_X+\Delta )|_{S^\nu}$, then $(S^\nu,\Delta _{S^\nu})$ is dlt and there is a bijection between the strata of $\lfloor \Delta _{S^\nu}\rfloor$ and the strata of  $\lfloor \Delta \rfloor$ that are contained in $S$. 

Let $P\in S$ be any codimension 1 point on $S$ and $m$ the Cartier index of $K_X$ at $P$. Then on a neighbourhood of $P\in X$,  $mD$ is Cartier for any divisor $D$, and $S$ is normal, and $(K_X+S)|_S=K_S+(1-\frac 1 m )P$.
\end{lemma}
\begin{proof} By assumption there is an open subset $U\subset X$ containing the generic points of every strata of $\lfloor \Delta \rfloor$ such that $(U,\lfloor \Delta \rfloor|_U)$ has simple normal crossings and the complement $Z=X\setminus U$ contains no non-klt centres.
Then, it is easy to see that $(S\cap U, \lfloor \Delta \rfloor|_{S\cap U})=(S^\nu\cap U,\lfloor \Delta _{S^\nu}\rfloor|_{S^\nu\cap U})$ has simple normal crossings. Let $E$ be an exceptional divisor over $S$ with centre $c_E$  whose generic point is not contained in $\lfloor \Delta \rfloor|_{S\cap U}$,  {and let $I$ be an ideal sheaf on $S$, the blow-up of which contains $E$. By blowing-up the pushforward of $I$ on $X$ and taking the normalisation, we may construct} a morphism  $X'\to X$ such that if $S'\subset X'$ is the strict transform of $S$, then $E$ is a divisor on $S'$.  After possibly further blow ups of $X$, we may assume that there is an exceptional divisor $F\subset X'$ such that $S'\cap F= E$ and $S',F$ intersect transversely at the generic point of $E$.
 But then the log discrepancies satisfy $a_E(S^\nu , \Delta _{S^\nu})=a_F(X,\Delta )>0$. The first part of the lemma now follows easily.
 
 The second part of the lemma follows by standard results for surfaces once we localise at $P\in X$.
\end{proof} 

We state an application of \cite{Witaszek2020KeelsTheorem}.
\begin{theorem}[{\cite[Theorem 2.22]{witaszek2021relative}}] \label{t-semiampless-universal-homeo} Let $f \colon Y \to X$ be a finite  universal homeomorphism of schemes which are proper over a Noetherian base scheme $S$. Let $\mathcal{L}$ be a nef line bundle on $X$ such that $\mathcal{L}|_{X_{\Q}}$ and $f^*\mathcal{L}$ are semiample, where $X_{\Q}$ is the characteristic zero fibre of $X \to \Spec \mathbb{Z}$. Then $\mathcal{L}$ is semiample.
\end{theorem}
\subsection{F-regularity and $+$-regularity}
Since the work of Hacon and Xu \cite{hx13}, (global) F-regularity has been one of the main tools in the minimal model program in positive characteristics. In the local setting it gives a good analog of (and often coincides with) klt singularities and in the global setting it provides us with a log Fano structure and implies vanishing theorems.  For the convenience of the reader we recall several definitions and key results.
\begin{defn} Let $X$ be an $F$-finite scheme defined over a  field of characteristic $p>0$. Given an effective $\Q$-divisor $B$, we say that $(X,B) $ is {\it globally F-split}  if for every integer $e>0$, the natural homomorphism of $\OO _X$-modules
\[\OO _X\to F_*^e\OO _X(\lfloor (p^e-1)B\rfloor )\] splits.
 We say that  $(X,B) $ is  {\it globally F-regular} (resp.\ {\it purely globally F-regular}) if for every effective divisor $D$ on $X$ (resp.\ every $D\geq 0$ intersecting $\lfloor B \rfloor$ properly) and every integer $e\gg 0$, the natural homomorphism of $\OO _X$-modules
\[\OO _X\to F_*^e\OO _X(\lfloor (p^e-1)B\rfloor +D)\] splits.
\end{defn}
If the above splittings hold locally on $X$, then we refer to these notions as {\it F-purity}, {\it strong F-regularity}, and {\it pure F-regularity}, respectively. Given a morphism $f\colon X\to Y$, we say that $(X,B)$ is relatively (over $Y$) F-split, F-regular, or purely F-regular, if the above splittings hold locally over $Y$.

Let $L$ be a  divisor and $(X,B)$ be a log pair where  $p$ does not divide the index of $B$. We let
\[
S^0(X,B;L):=\bigcap\, {\rm Im}\left( H^0(X,F^e_*\OO _X((1-p^e)(K_X+B)+p^eL))\to H^0(X, \OO _X(L))\right),
\] 
where the intersection is taken over $e>0$ sufficiently divisible so that $(p^e-1)B$ is integral.
{Note that if $f\colon X\to U$ is a projective morphism to an affine variety such that $L$ is $\Q$-Cartier and $L-(K_X+B)$ is $f$-ample,  then (identifying $f_*\OO_X(L)$ with $H^0(X,\OO _X(L))$) the subsheaf $S^0(X,B;L)\subset 
f_*\OO_X(L)$ is coherent and equal to the single image for a sufficiently big and divisible $e>0$ (see the proof of \cite[Proposition 2.15]{hx13}).}\\

We quickly review the theory of $+$-regularity (\cite{BMPSTWW20}). Let $X$ be a normal integral scheme which is proper over a \emph{complete} Noetherian local domain $(R,\mathfrak{m})$ with characteristic $p>0$ residue field. Let $B \geq 0$ be a $\Q$-divisor on $X$. For a Weil divisor $L$ on $X$, we define the subspace of $+$-stable sections $B^0(X, B; \OO_X(L))$ to be
\[
\bigcap_{\substack{f \colon Y \to X\\ \mathrm{finite}}}{\rm Im}\,\big(H^0(Y, \OO_Y( K_Y + \lceil{f^* (L - K_X - B)}\rceil)) \to H^0(X, \OO_X(L))\big), 
\]
where the intersection taken over all finite covers $f \colon Y \to X$ by a normal integral scheme $Y$ (more precisely, here and throughout the article, this means that we take the intersection over the category of all finite covers equipped with an embedding of $K(Y)$ into a fixed algebraic closure $\overline{K(X)}$ of $K(X)$, see \cite[Convention 4.1]{BMPSTWW20}). 
When $X$ is not irreducible, but still normal, we define $B^0(X, B; \OO_X(L))$ to be the direct sum of $B^0$ for each connected (thus irreducible) component of $X$.

Similarly, for a reduced divisor $S$ having no common components with an effective $\Q$-divisor $B$, we define the adjoint variant $B^0_{S}(X, S+B; \OO_X(L))$ to be equal to
\[
\bigcap_{\substack{f \colon Y \to X\\ \mathrm{finite}}}{\rm Im}\,\Big(H^0(Y, \bigoplus_{i=1}^t\OO_Y( K_Y +S_{i,Y} + \lceil{f^* (L - K_X-S-B)}\rceil)) \to H^0(X, \OO_X(L))\Big), 
\]
where $S_{i,Y}$ are compatibly chosen prime divisors lying over $S_{i}$ for $S = \sum_{i=1}^t S_i$ and $S_i$ prime. The choice of such divisors $S_{i,Y}$ is equivalent to choosing prime divisors $S_i^+$ in $X^+$ lying over $S_i$. Note that the definition of $B^0_S$ is independent of the choice of these $S_i^+$ (see \cite[Lemma 4.23]{BMPSTWW20}).


\begin{defn}\label{defn:global-plus-regularity} \label{defn:pure-global-plsu-regularity}
Let $X$ be a normal, integral, excellent scheme with a dualising complex and with every closed point having positive characteristic residue field. Further, let $B \geq 0$ be a $\Q$-divisor on $X$.

We say that $(X,B)$ is \emph{globally $+$-regular}, if for every finite dominant map $f \colon Y \to X$ with $Y$ normal, the morphism 
\[
\OO_X \to f_* \OO_Y(\lfloor f^*B \rfloor)
\]
splits. When $X$ is not integral, but still normal with all connected components of the same dimension, we say that $X$ is \emph{globally $+$-regular} if so are all of its connected components.

 {Let $(X,S+B)$ be a log pair such that $S$ is a reduced divisor having no common components with an effective $\Q$-divisor $B$.} We say that $(X,S+B)$ such that $\lfloor B \rfloor =0$, is \emph{purely globally $+$-regular} (along $S$), if for every finite dominant map $f \colon Y \to X$ with $Y$ normal, the morphism 
\[
\OO_X \to f_* \bigoplus_{i=1}^t\OO_Y(\lfloor f^*(S+B)\rfloor - S_{i,Y})
\]
splits where $S_{i,Y}$ are chosen as in the definition of $B^0_S$ above.
\end{defn} 

In what follows we recall the relative analogs of the above notions. 
\begin{defn}
Let $Z$ be an excellent scheme with  a dualising complex. Let $f \colon X \to Z$ be a proper morphism.

We say that $(X,B)$ for a $\Q$-divisor $B\geq 0$ is \emph{completely relatively $+$-regular} over $Z$ if for every closed point $z \in Z$ with positive residue characteristic its base change $(X_{\widehat{Z}_z}, B_{\widehat{Z}_z})$ to the completion of $\OO_{Z,z}$ at $z$  is globally $+$-regular.

We say that $(X,S+B)$ for a reduced divisor $S$ having no common components with a $\Q$-divisor $B\geq 0$ is \emph{completely relatively purely $+$-regular} over $Z$ if for every closed point $z \in Z$ with positive residue characteristic its base change $(X_{\widehat{Z}_z}, S_{\widehat{Z}_z} +  B_{\widehat{Z}_z})$ to the completion of $\OO_{Z,z}$ at $z$  is purely globally $+$-regular.
\end{defn}
\noindent Note that in contrast to \cite{BMPSTWW20}, we do not assume that $Z$ has all closed points of positive residue characteristics. In particular, our definition is \emph{meaningless} in a neighbourhood of \emph{closed} points of characteristic zero. However, such a formulation of the definition allows us to simplify some of the statements later on.

The scheme $X$ proper over $Z=\Spec R$ having all closed points of positive residue characteristics is globally $+$-regular if and only if it is completely relatively $+$-regular over $R$ (\cite[Corollary 6.9]{BMPSTWW20}). However, this is not known for pure global $+$-regularity (unless $H^0(X,\OO_X)=R$ and $-(K_X+S+B)$ is big and semiample, see \cite[Corollary 7.6]{BMPSTWW20}), and so for coherence of notation we shall always add the prefix \emph{completely} when talking about any of these notions in the relative setting.

We also point out to the reader that, with notation as above, if $X$ is proper over a complete Noetherian local domain $R$  of positive residue characteristic, then global $+$-regularity and pure global $+$-regularity of $(X,B)$ and $(X,S+B)$, respectively, are equivalent to $B^0(X,B; \OO_X) = H^0(X,\OO_X)$ and $B^0_{S}(X,S+B; \OO_X) = H^0(X,\OO_X)$  by \cite[Proposition 6.8 and Proposition 6.26]{BMPSTWW20}. 



\begin{lemma}\label{l-small} Let $X$ and $X'$ be two normal varieties defined over an $F$-finite field of characteristic $p>0$ and let $\phi \colon X\dasharrow X'$ be a birational map which is an isomorphism at any codimension one point. Let $B\geq 0$ be a $\Q$-divisor on $X$ and $B'=\phi _* B$ the corresponding $\Q$-divisor on $X'$. Then $(X,B)$ is globally F-split (resp.\ globally F-regular or purely globally F-regular) if and only if $(X',B')$ is globally F-split (resp.\ globally F-regular or purely globally F-regular).

The same result holds for global $+$-regularity and global pure $+$-regularity when $X$ and $X'$ are as in Definition \ref{defn:global-plus-regularity}.  
\end{lemma}
\begin{proof}
This follows from the fact that a splitting of structure sheaves for a finite map of normal varieties may checked on an open subset with complement of codimension at least two. In positive characteristic, this lemma has been observed in \cite[Proof of Proposition 6.3]{schwedesmith10}. In mixed characteristic,  the proof is analogous to \cite[Proposition 6.18]{BMPSTWW20}.
\end{proof}

\begin{proposition}\label{p-das} Suppose that $(X,S+B)$ is a purely $F$-regular pair over an $F$-finite field of characteristic $p>0$ (or a purely $+$-regular pair as in Definition \ref{defn:pure-global-plsu-regularity}) where $\lfloor S+B\rfloor =S$ is a prime divisor. Then $S$ is normal.
 \end{proposition}
\begin{proof} In positive characteristic see \cite[Theorem A]{MSTWW20} or the proof of \cite[Theorem A]{Das15}. In mixed characteristic, this is \cite[Corollary 7.9]{BMPSTWW20}. 
 \end{proof}

Further, recall the following result known as inversion of F-adjunction ($+$-adjunction, resp.).
\begin{lemma}\label{l-das}
Let $(X,S+B)$ be a plt pair, where $S=\lfloor S+B\rfloor $ is a prime divisor and $B\geq 0$ has no common components with $S$, and let $f \colon X\to Z$ be a projective birational morphism of normal varieties over an $F$-finite field of characteristic $p>0$ (resp.\ over a DVR of characteristic $(0,p)$ for $p>0$). Assume that $-(K_X+S+B)$ is $f$-ample and write $K_{S^\nu}+B_{S^\nu}=(K_X+S+B)|_{S^\nu}$ for the normalisation $S^\nu\to S$. 

Then $(X,S+B)$ is relatively purely F-regular (resp.\ completely relatively purely $+$-regular) over a neighbourhood of $f(S)\subset Z$ if and only if $(S^\nu,B_{S^\nu})$ is relatively F-regular (resp.\ completely relatively $+$-regular) over $f(S)$. Under these equivalent assumptions, $S$ is normal. \end{lemma}
\begin{proof} 
In positive characteristic this follows by the same proof as \cite[Lemma 2.10]{HW17}. It uses the equality between the different and the $F$-different. For this, note that \cite[Theorem B]{Das15} assumes that the base field is algebraically closed, but the proof goes through for every $F$-finite base field. 

In mixed characteristic, this is \cite[Corollary 7.5]{BMPSTWW20} (after completing or invoking  \cite[Corollary 7.6]{BMPSTWW20}).

Finally, the normality of $S$ around positive characteristic closed points follows from Proposition \ref{p-das}, and around characteristic zero closed points by standard results (\cite{KM98}).   
\end{proof}

We also have the following immediate consequence of \cite[5.3]{Schwede14}.
\begin{proposition}\label{p-ext} Let $f\colon X\to U$ be a projective morphism of normal varieties over an $F$-finite field of positive characteristic where $U$ is affine, $L$ a $\Q$-Cartier Weil divisor, and $(X,S+B)$ a log pair such that $\lfloor S+B\rfloor =S$ is normal integral, $p$ does not divide the index of $B$, and $L-(K_X+S+B)$ is $f$-ample. Suppose that $L$ is Cartier on an open neighbourhood of $S\,\backslash\, Z$, where $Z \subseteq S$ is a closed subset of codimension at least two. Further, assume that $X$ is $\mathbb{Q}$-factorial, strongly $F$-regular, and of dimension at least three.
Then $S^0(X,S+B;L)\to S^0(S,B_S; L|_S)$ is surjective where $(K_X+S+B)|_S=K_S+B_S$.
\end{proposition}
\noindent Since $Z \subseteq S$ is a closed subset of codimension at least two, $L|_S$ is well defined as a Weil divisor. If $L$ is Cartier, then the assumption that $X$ is  $\mathbb{Q}$-factorial, strongly $F$-regular, and of dimension at least three is not needed (cf. \cite[5.3]{Schwede14}).
\begin{proof} This follows easily from the following commutative diagram, which exists as the different is equal to the $F$-different
\begin{center}
\begin{tikzcd}
F^e_* \OO_X((1-p^e)(K_X+S+B) + p^eL) \arrow{d} \arrow{r} &  F^e_* \OO_S((1-p^e)(K_S+B_S) + p^eL|_S) \arrow{d}\\
\OO_X(L) \arrow{r} & \OO_S(L|_S).
\end{tikzcd}
\end{center}
Here, we use that $(X,S+B)$ is $\mathbb{Q}$-factorial, strongly $F$-regular, and of dimension at least three to guarantee that divisorial sheaves on $X$ are $S_3$ (cf.\ \cite[Proposition 2.7]{HW17}), which implies that the upper horizontal arrow is a surjection of sheaves (cf.\ \cite[Lemma 2.4]{HW17}). This is turn gives the surjection on global sections for $e \gg 0$ by Serre's vanishing.
\end{proof}
In the case of $B^0$, the assumption on $+$-regularity of singularities when $L$ is not Cartier is not needed. This discrepancy stems from $X^+$ being cohomologically Cohen-Macaulay in positive and mixed characteristic (informally speaking, the `non-Cohen-Macauliness' is killed by finite covers).
\begin{proposition} \label{p-ext-mixed}
Let $X$ be a projective normal integral scheme over a complete local Noetherian domain $(R, \mathfrak{m})$ of characteristic $(0,p)$ for $p> 0$. Let $L$ be a $\Q$-Cartier Weil divisor, and let $(X,S+B)$ a log pair such that $\lfloor S+B\rfloor =S$ is normal integral and $L-(K_X+S+B)$ is $f$-ample. Suppose that $L$ is Cartier on an open neighbourhood of $S\,\backslash\, Z$, where $Z \subseteq S$ is a closed subset of codimension at least two. Then $B^0_S(X,S+B;L)\to B^0(S,B_S; L|_S)$ is surjective where $(K_X+S+B)|_S=K_S+B_S$.
\end{proposition}
\noindent The assumption that $S$ is connected is unnecessary but we kept it for simplicity as in our applications the connectedness will be preserved under completion. In general, however, this is not always the case.
\begin{proof}
The proof is the same as in \cite[Theorem 7.2]{BMPSTWW20}. 
We recall it for the convenience of the reader. We refer to \cite[Subsection 2.1]{BMPSTWW20} for a primer on Matlis duality.

Set $\mathcal{N} = \OO_{X^+}(N)$ for $N = \pi^*(K_X+S+B-L)$, where $\pi \colon X^+ \to X$ is the natural map. Note that $\mathcal{N}$ is a line bundle, as every $\Q$-Cartier $\Q$-divisor is rendered Cartier by some finite cover. Let $S^+$ be the chosen prime divisor on $X^+$, lift $S$ and consider the following diagram wherein the left square exists as $B$ is effective:
{\small
\begin{center}
\begin{tikzcd}[column sep = tiny]
0 \arrow{r} & \OO_X(K_X - L) \arrow{d} \arrow{r} &  \OO_X(K_X+S-L) \arrow{r}\arrow{d} & \OO_X(K_X+S-L)\, /\, \OO_X(K_X-L) \arrow[dashed]{d} \arrow{r} & 0 \\
0 \arrow{r} & \pi_*\OO_{X^+}(N-S^+) \arrow{r} &  \pi_*\OO_{X^+}(N) \arrow{r}  & \pi_*(\OO_{S^+} \otimes \mathcal{N}) \arrow{r} & 0.
\end{tikzcd}
\end{center}
}
\noindent Here, the right vertical dashed arrow exists by an easy diagram chase. Moreover, this arrow factorises through the $S_2$-fification $\OO_{S}(K_{S}-L|_S)$ of the upper term as $S^+$ is normal, and so $\OO_{S^+} \otimes \mathcal{N}$ is $S_2$. Now, apply local cohomology with $d = \dim X$ to get:
{\small
\begin{center}
\begin{tikzcd}[column sep = tiny]
 H^{d-1} R\Gamma_{\mathfrak{m}} R\Gamma(S, \OO_X(K_X+S-L)\, /\, \OO_X(K_X-L)) \arrow[two heads]{d} \arrow{r} & H^d R\Gamma_{\mathfrak{m}} R\Gamma(X, \OO_X(K_X-L)) \arrow{dd}{(**)} \\
  H^{d-1} R\Gamma_{\mathfrak{m}} R\Gamma(S, \OO_{S}(K_{S}-L|_S)) \arrow{d}{(*)} &  \\
 H^{d-1} R\Gamma_{\mathfrak{m}} R\Gamma(S^+, \OO_{S^+} \otimes \mathcal{N}) \arrow[hook]{r} & H^d R\Gamma_{\mathfrak{m}} R\Gamma(X^+, \OO_{X^+}(N-S^+)).
\end{tikzcd}
\end{center}
}
\noindent Here, the upper left vertical arrow is surjective as the cokernel of 
\[
\OO_X(K_X+S-L)\, /\, \OO_X(K_X-L) \to \OO_{S}(K_{S}-L|_S)
\]
has codimension at least $1$ in $S$, and so the top local cohomology ($H^{d-1} R\Gamma_{\mathfrak{m}} R\Gamma$) thereof vanishes. The bottom horizontal arrow is injective, as 
\[
H^{d-1} R\Gamma_{\mathfrak{m}} R\Gamma(X^+, \OO_{X^+}(N))  = 0
\]
by \cite[Corollary 3.7]{BMPSTWW20}.

The upshot is that the image of $(*)$ injects into the image of $(**)$. Note that
\[
\OO_{S^+} \otimes \mathcal{N} \simeq \OO_{S^+}((\pi|_{S^+})^*(K_S+B_S - L|_S))
\]
by definition of the different (cf.\ the second part of \cite[Subsection 2.1]{MSTWW20}), and so by taking Matlis duality we get that:
\[
B^0_S(X,S+B;L)\to B^0(S,B_S; L|_S)
\]
is surjective. Here the left hand side is Matlis dual to the image of $(**)$ by \cite[Definition 4.21 and Lemma 4.24]{BMPSTWW20}, while the right hand side is Matlis dual to the image of $(*)$ by \cite[Lemma 4.8]{BMPSTWW20}. \qedhere

\end{proof}

\subsection{Special termination}\label{ss-ter}

In this section $X$ is a normal variety defined over a field $k$ of characteristic $p>0$ or a DVR $R$ of characteristic $(0,p)$ for $p>0$. Recall the following result known as special termination \cite[Theorem 4.2.1]{fujino05}.
\begin{theorem}\label{t-specialter} Assume that the log minimal model program for $\Q$-factorial dlt pairs holds in dimension $\leq n-1$ (including the termination of all flips). Let $X$ be a normal $\Q$-factorial $n$-dimensional variety, let $B$ be an effective $\mathbb R$-divisor such that $(X,B)$ is dlt, and let \[(X,B)\dasharrow (X_1,B_1) \dasharrow (X_2,B_2) \dasharrow \ldots\]
be a sequence of $(K_X+B)$-flips. Then after finitely many steps, the flipping loci are disjoint from $\lfloor B_i \rfloor$.
\end{theorem}
Since the minimal model program for surfaces is known in full generality, the above result implies special termination for threefolds and in some special cases for fourfolds over perfect $F$-finite fields of characteristic $p>0$. Note that it is not true that in positive characteristic log canonical centres are normal. However, by Lemma \ref{l-dltad}, we know that  if $W^\nu\to W$ is the normalisation of a log canonical centre of a dlt pair $(X,B)$, then the induced pair $K_{W^\nu}+B_{W^\nu}=(K_X+B)|_{W^\nu}$ is also dlt, its log canonical centres are in bijection with the log canonical centres of $(X,B)$ contained in $W$, and the coefficients of $B_{W^\nu}$ are the same as those given by the usual adjunction in characteristic 0.
\begin{theorem}\label{t-st3} The statement of Theorem \ref{t-specialter} holds in dimension three over all {F-finite} fields $k$ of characteristic $p>5$ or DVRs $R$ of characteristic $(0,p)$ for $p>5$. 

\end{theorem}
\begin{proof} The three dimensional case of Theorem \ref{t-specialter} is an immediate consequence of the proof of \cite[Theorem 4.2.1]{fujino05} and the two-dimensional minimal model program \cite{tanaka16_excellent}.  

\end{proof}

\begin{theorem}\label{t-st4} Let  $(X,B)$ be a $\Q$-factorial four-dimensional dlt pair defined over a field of characteristic $p \geq 0$ or a DVR of characteristic $(0,p)$ for $p>0$, where $B$ is an effective $\mathbb R$-divisor. Let $\pi \colon X\to U$ be a projective morphism and suppose that \[(X,B) =: (X_0,B_0) \dasharrow (X_1,B_1) \dasharrow (X_2,B_2) \dasharrow \ldots\]
is a sequence of $(K_X+B)$-flips and divisorial contractions over $U$. Then after finitely many steps, the intersections of the flipping and flipped loci with the non-klt locus are at most one dimensional. In particular, the flipping and flipped loci for $X_i\dasharrow X_{i+1}$ cannot be both contained in $\lfloor B_i \rfloor$ and $\lfloor B_{i+1} \rfloor$.

Moreover, if $\dim U=1$, $\kappa (K_X+B/U)\geq 0$,  and
$\lfloor B \rfloor$ contains the fibre $X_u$ for some point $u\in U$, then after finitely many steps, the flipping loci are disjoint from $X_u$. 
\end{theorem}
\begin{proof} 

Since divisorial contractions descrease the Picard rank, we may assume that the above sequence consists only of flips. Denote the flips by $\phi_i \colon X_i \dashrightarrow X_{i+1}$ and denote the flipping and the flipped contractions by $f_i \colon X_i \to Z_i$ and $f_i^+ \colon X_{i+1} \to Z_i$, respectively. Let $S$ be an irreducible component of $\lfloor B \rfloor$, let $S_i$ be its strict transform on $X_i$, and let $T_i := f_i(S_i)$. Note that the normalisation of $S_i$ need not be $\Q$-factorial.

 By \cite[Proposition 4.2.14]{fujino05} we may assume that $f^+_i|_{S_{i+1}} \colon S_{i+1} \to T_i$ is small. Moreover, if $f_i|_{S_i} \colon S_i \to T_i$ is divisorial, then the Picard rank as defined in \cite[Lemma 1.6]{AHK07} satisfies $\rho (S^\nu_i/U)>\rho(S^\nu_{i+1}/U)$, where $S^{\nu}_i$ and $S^{\nu}_{i+1}$ are appropriate normalisations. Thus again we may assume $f_i|_{S_i}$ is small, i.e.\ $\phi_i|_{S_i} \colon S_i \dashrightarrow S_{i+1}$ is of type (SS). The sum of the dimensions of the flipping and the flipped locus of $X_i\dasharrow X_{i+1}$  is at least three, so one of them must be 2-dimensional (in fact they must be of type $(2,1)$, $(2,2)$ or $(1,2)$), and so they cannot be both contained in the non-klt locus.\\

Now, assume that $\dim U=1$, $\kappa(K_X+B/U)\geq 0$, and $\lfloor B \rfloor$ contains the support of the fibre $X_u$ for some point $u\in U$.
As observed above, after finitely many steps,  $\lfloor B_i\rfloor$ contains no two-dimensional component of the flipping or flipped loci. 
 Suppose that the flipping locus (and thus also the flipped locus) is contained in the fibre over $u$ and hence in the non-klt locus. Since either the flipping or flipped locus is two dimensional,  this can not happen infinitely often. Therefore, we may assume that for $i\gg 0$ the flipping and flipped loci dominate $U$.
 
 Therefore, we can replace $(X,B)$ and $(X_i,B_i)$ by $(X_{\eta}, B_{\eta})$ and $((X_i)_{\eta}, (B_i)_{\eta})$, respectively, so that these are three-dimensional $\Q$-factorial pairs defined over an $F$-finite field $k(\eta)$ for the generic point $\eta \in U$. In particular, we can assume that log resolutions and terminalisations exist (see \cite{DW19} when $\eta$ is of positive characteristic and \cite{bchm06} when it is of characteristic zero).
 
 First we will show that $\lfloor B_i \rfloor$ is disjoint from the flipping, and so the flipped, locus. By what we have proven above, we may also assume that the intersection of the non-klt locus with the flipping and the flipped locus is at most zero-dimensional. Now, suppose that a flipping curve $C$ intersects $\lfloor B_i\rfloor$. Then $C \cdot \lfloor B_i \rfloor > 0$, and so the flipped locus in contained in $\lfloor B_{i+1} \rfloor$ which is impossible. 
 Now, replacing $B_i$ by $\{B_i\}$, we may assume that $(X_i,B_i)$ are klt, and so the mmp should terminate by Proposition \ref{proposition:AHK} for $U = \Spec k(\eta)$.
 
 However, the above replacement may render the assumption $\kappa(K_X+B) \geq 0$ invalid, so we need to be more careful. Precisely, write $K_X+B\sim _{\Q}M\geq 0$ where the supports of the strict transforms $M_i$ on $X_i$ contain the flipping loci for all $i$. Since the flipping loci are disjoint from $\lfloor B_i \rfloor$, the above $(K_X+B)$-mmp (equivalently, a $(K_X+B+\epsilon M)$-mmp), is also a $(K_X+\{B\} + \epsilon M)$-mmp for $0 < \epsilon \ll 1$. Since the flipping loci are contained in $\Supp(\{B\}+ \epsilon M)$, this mmp terminates by  Proposition \ref{proposition:AHK}.

\qedhere

\end{proof}

\begin{proposition} \label{proposition:AHK} Let  $(X,B)$ be a $\Q$-factorial klt log pair of dimension $n$ defined over a field of characteristic $p\geq 0$  where $B$ is an effective $\mathbb R$-divisor. Let $\pi \colon X\to U$ be a projective morphism and suppose that \[(X,B) =: (X_0,B_0) \dasharrow (X_1,B_1) \dasharrow (X_2,B_2) \dasharrow \ldots\]
is a sequence of $(K_X+B)$-flips over $U$.

Assume that log resolutions of all $n$-dimensional log pairs with the underlying variety being birational to $X$ exist, and so do terminalisations for $n$-dimensional $\Q$-factorial klt pairs. Then it can happen only finitely many times that the flipping or the flipped locus has a component of codimension two in $X_n$ and which is contained in $\Supp\, B_n$. In particular, if $\kappa(K_X+B/U) \geq 0$, then the flipping locus is of codimension at least three for $n \gg 0$.
\end{proposition}
\noindent This result holds in a more  general setting (that is, as in \cite{kollar13} or \cite{BMPSTWW20}).
\begin{proof}
This follows by the same argument as  \cite[Theorem 2.15]{AHK07}. For the convenience of the reader, we recall the argument briefly.
Define the following functions $w^{-}_{\alpha}, w^{+}_{\alpha}, W^{-}_{\alpha}, W^+_{\alpha} \colon (-\infty, 1) \to \mathbb{R}_{\geq 0}$ for $\alpha \in (0,1)$ by
\begin{align*}
w_{\alpha}^-(x) &= 1-x \text{ for } x \leq \alpha, \text{ and } w_{\alpha}^-(x) =0 \text{ otherwise, }\\
w_{\alpha}^+(x) &= 1-x \text{ for } x < \alpha, \text{ and } w_{\alpha}^+(x) =0 \text{ otherwise, }\\
W_{\alpha}^-(b) &= {\textstyle \sum_{k=1}^{\infty}} w_{\alpha}^-(k(1-b)), \\
W_{\alpha}^+(b) &= {\textstyle \sum_{k=1}^{\infty}} w_{\alpha}^+(k(1-b)).
\end{align*}
For any sub-terminal log pair $(Y,B_Y)$ and $B_Y = \sum b_i B_i$ with $B_i$ irreducible, we define the difficulty (\cite[Definition 2.3]{AHK07})
\[
d^+_{\alpha}(Y,B_Y) = \sum_{b_i \leq 0} W^{+}_{\alpha}(b_i) \rho(\tilde B_i) + \sum_{v} w_{\alpha}^+(a_v) - \sum_{\tilde C_{i,j}} W^{+}_{\alpha}(b_i), \text{ where}\] 
\begin{itemize}[leftmargin=*]
	\item $\rho$ is as in \cite[Lemma 1.6]{AHK07} (in particular, stable under small birational maps), 
	\item $v$ is taken over all valuations which are not echoes on $Y$ (here, $a_v$ denotes the discrepancy of $(Y,B_Y)$ at $v$), and 
	\item $\tilde C_{i,j}$ are codimension one irreducible components in $\tilde B_i$  of the inverse image of a codimension two integral subscheme $C \subseteq Y$ under the normalisation $\bigsqcup \tilde B_i \to \bigcup B_i$. Here, we exclude those $C$ which are centres of echoes.
\end{itemize}
We say that an irreducible $C \subseteq Y$ is a \emph{centre of an echo} if it is of codimension two, contained in exactly one $B_i$, but not in $\mathrm{Sing}\, B_i$. Here, a valuation is called a \emph{($k$-th) echo} for such $C$ if it corresponds to the last exceptional divisor of a sequence of $k$ blow-ups at strict transforms of $C$ (see \cite[Example 1.4]{AHK07}). The discrepancy of the $k$-th echo is $k(1-b_i)$.  The difficulty is well defined and finite, because there are only finitely many valuations with discrepancy in $(-1,0]$, and also finitely many in $(0,1)$ if we exclude echoes (\cite[Lemma 1.5]{AHK07}; here we use log resolutions). 

We can similarly define $d^-_{\alpha}$. The difficulties are stable under log pullbacks (\cite[Lemma 2.7]{AHK07}), and so we extend these definitions to arbitrary $\Q$-factorial klt pairs by taking the difficulty of a terminalisation (\cite[Definition 2.8]{AHK07}). The difficulties are non-negative for klt pairs (\cite[Lemma 2.9]{AHK07}) and decreasing for a sequence of flips $(X_n,B_n) \dashrightarrow (X_{n+1},B_{n+1})$ if we start with $n \gg 0$ (\cite[Theorem 2.12]{AHK07}).

We are ready to give the proof of the proposition. First, we show that there are only finitely many flips with flipping or flipped loci admitting a component $C$ of codimension two and contained in $\mathrm{Sing}\, X_n$ (this is \cite[Lemma 2.14]{AHK07} with the use of Bertini replaced by localisation at $C$). By taking $n\gg 0$, we can assume that the places with discrepancies smaller or equal to zero are stable under flips: let $E_1, \ldots, E_m$ be such places over $X_n$ with discrepancies $a_1, \ldots, a_m \leq 0$. In particular, the minimal discrepancy of the two-dimensional singularity obtained by localisation at $C$ must be equal to $a_i$ for some $i$. It is known by experts that mlds for klt surface singularities satisfy the ascending chain condition, and since flips increase discrepancies, the statement follows. As this result is unpublished, we refer to Lemma \ref{l-acc-for-mlds} instead.

We are left to show that there are only finitely many flips with flipping or flipped loci admitting a component $C$ of codimension two and contained in $\Supp\, B_n$ but not $\mathrm{Sing}\, X_n$. The blow-up along $C$ produces a divisor with discrepancy $1 - \sum m_i b_i$ for $m_i \in \mathbb{Z}_{\geq 0}$. There are only finitely many such discrepancies, and, thus, it is enough to prove the following claim: given $\alpha \in (-1,1)$ there cannot be infinitely many flips for which there exists a valuation $v$ with $a_v = \alpha$ and centre contained in the flipping or the flipped locus (\cite[Theorem 2.13]{AHK07}). The case $\alpha \in (-1,0]$ follows by finiteness, while for $\alpha \in (0,1)$ we notice that $d^-_{\alpha}$ drops by at least $1-\alpha$ if the flipping locus admits such a valuation, and so does $d^+_{\alpha}$ for the flipped locus. Since the difficulties are positive and finite, this can only happen finitely many times. 

The last assertions follows by replacing $B$ by $B + \epsilon M$ for $0 \leq M \sim_{\Q, U} K_X + B$ and $0 < \epsilon \ll 1$, since then every flipping curve $\Sigma$ satisfies $\Sigma \cdot M < 0$, and so is contained in $M$.
\end{proof}

\begin{lemma} \label{l-acc-for-mlds} Fix $m \in \mathbb{N}$ and $\epsilon>0$. Consider all $\epsilon$-lc excellent two-dimensional pairs $(S,B)$ with the minimal log resolution having at most $m$ distinct discrepancies. Then the set of all these discrepancies for such surfaces is finite. 
\end{lemma}
\noindent \noindent This result holds in a more  general setting (that is, as in \cite{kollar13} or \cite{BMPSTWW20}).
\begin{proof} Let $f \colon S'\to S$ be the minimal resolution at a point $s \in S$ with $k$ exceptional curves $C_i$ (which are $k(s)$-schemes) and let $\Gamma$ be  the corresponding graph. Write $K_{S'} + B_{S'} = f^*(K_S+B_S)$  and let $a_i$ be the log discrepancy of $C_i$ in $(S',B_{S'})$. We will use the results of \cite{kollar13}. Set $r_i:=\dim _{k(s)}H^0(C_i, \mathcal O _{C_i})$. Then $C_i^2=-r_ic_i$ for $c_i \in \mathbb{Z}$. Note that $C_i$ are conics over the field $H^0(C_i, \mathcal{O}_{C_i})$ and in particular $\deg \omega _{C_i}=-2r_i$ (\cite[Reduction 3.30.2]{kollar13}). By adjunction,
\[
a_iC^2_i = (K_{S'} + B_{S'} + a_iC_i) \cdot C_i  \geq -2r_i,
\]
and so $-c_i \leq \frac{2}{a_i}$ is bounded. By classification (\cite[3.31 and 3.41]{kollar13}), $r_i \leq 4$, and so $-C_i^2$ is bounded. Further, $C_i \cdot C_j$ is bounded as well for every $i\neq j$ (see \cite[3.41]{kollar13}).

Recall that $\Gamma$ is a tree with at most one fork and three legs (\cite[3.31]{kollar13}), there are no $-1$-curves, and the convexity of discrepancies holds (\cite[Proposition 2.37]{kollar13}), that is, given three consecutive curves $C_1$, $C_2$, and $C_3$ on a leg, we have $2a_2\leq a_1+a_3$. In fact, it follows easily from the same proof that if $2a_2= a_1+a_3$, then $c_2=-2$, $f^{-1}_*B\cdot C_2=0$, and $C_2 \cdot C_1 = C_2 \cdot C_3 = r_2$.

By the negativity lemma, the $k$ equations 
\[
C^2_i = -2r_i + B_{S'}\cdot C_i
\]
are linearly independent and uniquely determine the discrepancies.
Given a maximal sequence of successive curves $C_1,\ldots, C_r$ on one leg with the same discrepancies $a_i$ and parameters $r_i$, we remove the equations corresponding to $C_2, \ldots, C_{r-1}$ and replace all occurrences of $a_2,\ldots, a_r$ by $a_1$. This operation does not change the set of solutions (indeed, this is equivalent to adding equations $a_1=a_2=\ldots=a_r$ which does not change the set of solutions but renders the equations corresponding to $C_2, \ldots, C_{r-1}$ trivial), and so, eventually, given the shape of $\Gamma$ (\cite[3.31]{kollar13}), we are left with bounded-in-$m$ number of equations with bounded coefficients and in the same or lower number of variables. Therefore the set of discrepancies is bounded in terms of $m$ as well.\qedhere


\end{proof}

\begin{remark} It is expected that one can prove the special termination 
of sequences of the fourfold mmp with scaling. We do not pursue it here.\end{remark}
\section{Extending sections}
In this section we prove an extension result for purely F-regular (and purely $+$-regular) pairs. In particular, the following proposition implies Theorem \ref{t-flips_exist+}. Our original arxiv submission covered the positive characteristic case only. Its generalisation to mixed characteristic was then obtained in \cite[Proposition 3.28]{TY20}. For the convenience of the reader, we append the proof of the mixed characteristic case at the end of the argument  using the techniques of \cite{BMPSTWW20}.

\begin{proposition}\label{p-1}Let $(X,S+A+B)$ be a $\Q$-factorial dlt pair of dimension at least three defined over an $F$-finite field of characteristic $p>0$ (resp.\ over a DVR of characteristic $(0,p)$ for $p>0$) where $\lfloor S+A+B\rfloor =S+A$, the $\Q$-Cartier Weil divisor $A$ is ample, and the Weil divisor $S$ is irreducible. Further, let $f\colon X\to Z$ be a contraction with $Z$ an affine variety such that $(X,S+(1-\epsilon)A+B)$ is relatively purely F-regular (resp.\ completely  relatively purely $+$-regular) over $Z$ for any $\epsilon >0$. Write $K_S+B_S = (K_X+S+B)|_S$ and $A_S = A|_S$.

In the mixed characteristic case, we also assume that $-(K_X+S+B+A)$ is relatively ample and that closed points of $Z$ are of positive residue characteristic. Then for every $k\geq 1$ such that $k(K_X+S+B+A)$ is Cartier, we have \[|k(K_X+S+A+B)|_S=|k(K_S+A_S+B_S)|.\]
\end{proposition}
\noindent Note that $S$ is normal by Lemma \ref{l-das}. Also observe that the assumption on $\Q$-factoriality and the dimension being at least three is unnecessary in mixed characteristic. Further, the statement is in fact true over characteristic zero closed points by Remark \ref{r-xu} or by an analogous argument with $S^0$ replaced by $H^0$ using Nadel vanishing.
\begin{proof} \label{r-different-proof-p-1}
We tackle the positive characteristic case first. Let $F \in |k(K_S+B_S+A_S)|$. We will construct divisors $G_m \in |k(K_X+S+B+A) + mA|$ such that $G_m|_S = F + mA_S$ by descending induction. First, we take $M \gg 0$ for which we can construct $G_M$ by Serre vanishing. Now, assume that we constructed $G_{m+1}$ as above. To construct $G_m$ we proceed as follows. Write
\begin{align*}
L &= k(K_X+S+B+A) + mA \\
  &= K_X+S+B + (k-1)(K_X+S+B+A) + (m+1)A \\
  &\sim_{\Q} K_X+S+B + {\textstyle \frac{k-1}{k}}G_{m+1} + {\textstyle \frac{m+1}{k}}A.
\end{align*}
Since $L-(K_X+S+B + \frac{k-1}{k}G_{m+1})$ is ample, up to a small perturbation of the coefficients, by Proposition~\ref{p-ext} we have a surjection (the assumption on the non-Cartier locus of $A$, and hence of $L$, is satisfied as $(X,S+A+B)$ is dlt)
\[
S^0(X,S+B+{\textstyle \frac{k-1}{k}}G_{m+1}; L) \to S^0(S,B_S + {\textstyle \frac{k-1}{k}}G_{m+1}|_S; L|_S).
\]
Since $(S,B_S + \frac{k-1}{k}A_S)$ is relatively F-regular and 
\[
F + mA_S \geq {\textstyle \frac{k-1}{k}}G_{m+1}|_S - {\textstyle \frac{k-1}{k}}A_S = {\textstyle \frac{k-1}{k}}(F + mA_S),
\]
Lemma \ref{l-1} implies that the section corresponding to $F + mA_S$ lies in the image of the above map. This concludes the construction of $G_m$, and so we obtain a lift $G_0 \in |k(K_X+S+B+A)|$ of $F$.

In mixed characteristic, since surjectivity of maps of coherent sheaves may be checked on completions of closed points, we can complete at a closed point $z \in Z$ of positive residue characteristic and assume that $Z = \Spec R$ for a complete Noetherian local domain $R$ (note that this need not preserve $\Q$-factoriality). By \cite[Corollary 7.8]{BMPSTWW20} (and \cite[Corollary 7.5]{BMPSTWW20}), $S$ stays prime after the base change to the completion (here we use that $-(K_X+S+B+A)$ is relatively ample). Now the proof follows the positive characteristic case word for word with $F$-regular replaced by $+$-regular, Proposition \ref{p-ext} replaced by Proposition \ref{p-ext-mixed}, and the surjectivity of $S^0$ replaced by  the surjectivity of
\[
B^0_{S}(X,S+B+{\textstyle \frac{k-1}{k}}G_{m+1}; L) \to B^0(S,B_S + {\textstyle \frac{k-1}{k}}G_{m+1}|_S; L|_S). \qedhere
\]
\end{proof}
\begin{lemma} \label{l-1}Let $(X,B)$ be a projective globally F-regular pair defined over an affine variety over an $F$-finite field of characteristic $p>0$ (resp.\ a globally $+$-regular pair defined over a complete local Noetherian domain of characteristic $(0,p)$ for $p>0$), let $L$ be a Weil divisor, and let $
\Gamma$ be an effective $\Q$-divisor. If $g\in H^0(X, \OO_X(L))$ corresponds to a divisor $G\in |L|$ such that $G\geq \Gamma$, then $g\in S^0(X, B+\Gamma; L)$ (resp.\ $g\in B^0(X, B+\Gamma; L)$).
\end{lemma}
\begin{proof} First, we deal with the positive characteristic case. By replacing $X$ by its regular locus, we can assume that it is regular and hence factorial. Further, we perturb $B$ so that its index is not divisible by $p$. Since $(X,B)$ is globally F-regular, we have a splitting of
\[F^e_*\OO _X((1-p^e)(K_X+B))\to \OO _X,\]
and so $H^0(X, \OO_X) = S^0(X,B;\OO_X)$.

Thus, we have the following diagram
\begin{center}
\begin{tikzcd}
H^0(X,F^e_*\OO _X((1-p^e)(K_X+B))) \arrow[two heads]{d} \arrow{r}{\cdot g} & H^0(X,F^e_*\OO _X((1-p^e)(K_X+B)+p^eL)) \arrow{d} \\ 
H^0(X, \OO_X) \arrow{r}{\cdot g} & H^0(X, \OO_X(L)).
\end{tikzcd}    
\end{center}
 Therefore, $g$ is in the image of the map on global sections induced by the homomorphism
\[ H^0(X,F^e_*\OO _X((1-p^e)(K_X+B)+p^e(L-G)))\to H^0(X,\OO _X(L)).\]
Since $G\geq \Gamma$, this homomorphism factors through 
the homomorphism
\[H^0(X, F^e_*\OO _X((1-p^e)(K_X+B+\Gamma )+p^eL))\to H^0(X, \OO _X(L)),\]
concluding the proof. Indeed, the above map comes from taking the Grothendieck dual of $\OO_X(-G) \to F^e_*\OO_X((p^e-1)(B+\Gamma) - p^eG) \to F^e_*\OO_X((p^e-1)B)$.\\

In mixed characteristic, take a finite cover $f \colon Y \to X$ with $Y$ normal. By global $+$-regularity of $(X,B)$, the trace map
\[
\mathrm{Tr} \colon f_*\OO_Y(K_Y + \lceil-f^*(K_X+B)\rceil) \to \OO_X
\]
is split surjective (see the proof of \cite[Proposition 6.8]{BMPSTWW20}). Hence, by $B^0(X,B; \OO_X) = H^0(X, \OO_X)$ (see also \cite[Lemma 6.11]{BMPSTWW20}). By the same argument as above, $g$ is in the image of the map on global sections induced by
\[
H^0(Y, \OO_Y(K_Y + \lceil f^*(L-K_X-B-G)\rceil)) \to H^0(X,\OO_X(L)),
\]
and as $G \geq \Gamma$, this homomorphism factors through
\[
H^0(Y, \OO_Y(K_Y + \lceil f^*(L-K_X-B-\Gamma)\rceil)) \to H^0(X,\OO_X(L)),
\]
concluding the proof.
\end{proof}
\begin{remark} \label{r-xu} (C.\ Xu) Let $k$ be a field of characteristic zero or any field of positive characteristic, respectively. Suppose that $f\colon X\to Z$ is a flipping contraction and $\phi \colon X\dasharrow X^+$ is the flip of a dlt pair $(X,S+A+B)$, where both $-S$ and $A$ are $f$-ample Weil divisors, $S$ is irreducible, and $\lfloor S+A+B\rfloor = S+A$. Write $K_S+B_S = (K_X+S+B)|_S$ and $A_S=A|_S$.  Assume for simplicity that $S$ is normal.

We claim that $|k(K_X+S+A+B)|_S=|k(K_S+A_S+B_S)|$ for every integer $k>1$ such that $k(K_X+S+A+B)$ is integral (resp.\ any sufficiently divisible integer). This explains why we can obtain such strong liftability results, which are normally false when there is no ample divisor in the boundary (cf.\ \cite{HMK10}).

To see the claim, notice that if $S^+=\phi _*S$, then $S\dasharrow S^+$ extracts no divisors. In fact, if $P\subset S^+$ is an extracted divisor, then since $A^+=\phi _* A$ is $f^+\colon X^+\to Z$ anti-ample, it contains $P$ and hence there is a divisor $F$ over $X^+$ with centre $P$ such that the log-discrepancy $a_F(X^+,S^++A^++B^+)=0$ which is impossible, as $(X,S+A+B)$ is dlt and discrepancies improve after flips. 

Now, since $K_{X^+}+S^++A^++B^+$ is $f^+$-ample, by Kawamata-Viehweg vanishing (resp. Serre vanishing) 
\[ H^0(X^+, k(K_{X^+}+S^++A^++B^+))\to H^0(S^+, k(K_{S^+}+A_{S^+}+B_{S^+}))\] 
is surjective. Since $\phi$ is a small birational morphism, 
$
H^0(X^+, k(K_{X^+}+S^++A^++B^+))\cong H^0(X, k(K_{X}+S+A+B))
$
and since $S\dasharrow S^+$ is a $(K_S+A_S+B_S)$-non-positive birational contraction, 
\[
H^0(S^+, k(K_{S^+}+A_{S^+}+B_{S^+}))\cong H^0(S, k(K_S+A_S+B_S)).
\]
\end{remark}
\section{Relative Minimal Model Program over $\Q$-factorial fourfolds}
In this section we prove Theorem \ref{t-mmp+}.
\subsection{Existence of pl-flips with ample divisor in the boundary}
In order to show Theorem \ref{t-sflip}, we need the following lemma.
\begin{lemma}\label{l-2}
Let $(S,C+B)$ be a three-dimensional plt pair with standard coefficients defined over an $F$-finite field of characteristic $p>5$ (resp.\ a DVR of characteristic $(0,p)$ for $p>5$), where $f\colon S\to T$ is a projective birational morphism, $C$ is a prime divisor with $f|_{C} \colon C \to f(C)$ birational, and $-(K_S+C+B)$ is an $f$-ample $\Q$-divisor. 

Then $(S,C+B)$ is relatively purely F-regular (resp.\ completely relatively purely $+$-regular) over a neighbourhood of $f(C)\subset T$.

\end{lemma}
\begin{proof} 

Let $\tilde{C}$ be the normalisation of $C$ and write $K_{\tilde C} + B_{\tilde C} = (K_S+C+B)|_{\tilde C}$. Then $(\tilde C, B_{\tilde C})$ is klt with standard coefficients (see Lemma \ref{l-dltad}). Since $\tilde C \to f(C)$ is birational, $(\tilde C, B_{\tilde C})$ is relatively globally $F$-regular (\cite[Theorem 5.1]{daswaldron}, cf.\ \cite[Theorem 3.1]{hx13}, \cite[Proposition 2.9]{HW19b}), or completely relatively globally $+$-regular (\cite[Theorem 7.14]{BMPSTWW20}), respectively. Lemma \ref{l-das} implies that $(S,C+B)$ is relatively purely F-regular (resp.\ completely relatively purely $+$-regular) over a neighbourhood of $f(C)$.

\end{proof}  

\begin{theorem}\label{t-sflip} Let $(X,\Delta )$ be a four-dimensional $\Q$-factorial dlt pair with standard coefficients defined over an $F$-finite field $k$ of characteristic $p>5$ (resp.\ a DVR of characteristic $(0,p)$ for $p>5$) and $\phi \colon X\to Z$ a flipping contraction of a $(K_X+\Delta)$-negative extremal ray $R$ with $\rho(X/Z)=1$.
Suppose that there exist irreducible divisors $S,A\subset \lfloor \Delta \rfloor$ such that $R\cdot S<0$ and $R\cdot A>0$. 

Then the flip $(X^+,\Delta ^+)$ exists. Moreover, both $X$ and $X^+$ are relatively F-regular (resp.\ completely relatively $+$-regular) over a neighbourhood of the image of the flipping locus in $Z$.
\end{theorem}
\begin{proof} 
The statement is local and hence we may assume that $Z$ is affine and we work in a neighbourhood of some closed point $z\in Z$ of positive residue characteristic in the image of the flipping locus. Replacing $\lfloor \Delta \rfloor-S-A$ by $(1-\frac 1m)(\lfloor \Delta \rfloor-S-A)$ for some $m\gg 0$, we may assume that 
$\lfloor \Delta \rfloor=S+A$. Let $K_{S^{\nu}}+A_{S^{\nu}}+B_{S^{\nu}}=(K_X+S+A+B)|_{S^{\nu}}$ with $A_{S^{\nu}} = A|_{S^{\nu}}$ and $S^{\nu}$ being the normalisation of $S$. Then $(S^{\nu},A_{S^{\nu}}+B_{S^{\nu}})$ is a plt
threefold with standard coefficients. Since $A_{S^{\nu}}$ is relatively ample,  it is not exceptional, and so it is birational over its image in $Z$. As $-(K_{S^{\nu}}+A_{S^{\nu}}+B_{S^{\nu}})$ is relatively ample, Lemma \ref{l-2} implies that $(S^{\nu},A_{S^{\nu}}+B_{S^{\nu}})$ is purely F-regular (resp.\ completely purely $+$-regular)  over a neighbourhood of $f(A_{S^{\nu}})$. By Lemma \ref{l-das}, $(X,S+(1-\epsilon)A+B)$ is   purely F-regular (resp.\ completely purely $+$-regular) over a neighbourhood of $f(A_S)$ for any $\epsilon >0$. Since every flipping curve is contained in $S$ and intersects $A$, $(X,S+(1-\epsilon)A+B)$ is purely F-regular (resp.\ completely purely $+$-regular) over a neighbourhood of the image of the flipping locus. In particular, $S$ is normal (Lemma \ref{l-das}).

By Shokurov's reduction to pl-flips (see eg.\ \cite[Lemma 2.3.6]{cortibook}), it suffices to show that the restricted algebra \[R_S(K_X+S+A+B):={\rm Im}\left( R(K_X+S+A+B)\to R(K_S+A_S+B_S)\right)\]
is finitely generated.
By Proposition \ref{p-1} in fact $R_S(K_X+S+A+B)=R(K_S+A_S+B_S)$ (in divisible enough degrees). Since $(S,(1-\epsilon)A_S+B_S)$ is a klt threefold for $0 < \epsilon \leq 1$, the ring $R(K_S+A_S+B_S)$ is finitely generated (see \cite[Theorem 1.4 and 1.6]{DW19} when $S$ is purely positive characteristic and \cite[Theorem F and G]{BMPSTWW20} when $S$ is of mixed characteristic).

If $\phi \colon X \dashrightarrow X^+$ is the corresponding flip, then since the flipping contraction $f \colon X \to Z$ is relatively F-regular (resp.\ completely relatively $+$-regular) over a neighbourhood of the image of the flipping locus and $\phi$ is an isomorphism in codimension two, the flipped contraction $f^+ \colon X^+ \to Z$ is also relatively F-regular (resp.\ completely relatively $+$-regular), see Lemma \ref{l-small}. 
\end{proof}

\begin{remark}
The same argument shows that Lemma \ref{l-2} holds in dimension two for every $F$-finite field of characteristic $p>0$ without the assumption that $B$ has standard coefficients, and so Theorem \ref{t-sflip} holds in dimension three for $p>0$ without the assumption that $\Delta$ has standard coefficients (the required result on the canonical ring being finitely generated in dimension two follows from \cite[Theorem 1.1 and 4.2]{tanaka16_excellent}). This gives a direct proof of the existence of flips needed in \cite{HW19a} without referring to \cite{hx13}.
\end{remark}

\subsection{Proof of Theorem \ref{t-mmp+}}
We start by constructing contractions.
\begin{proposition}\label{p-cont} Let $(X,S+B)$ be a $\Q$-factorial dlt fourfold defined over perfect field of characteristic $p>5$ (or a DVR of characteristic $(0,p)$ for $p>5$ with perfect residue field) and equipped with a projective morphism $\pi \colon X\to U$ to a quasi-projective variety. Let $\Sigma$ be a $(K_X+S+B)$-negative extremal ray over $U$ such that $S\cdot \Sigma <0$. Further, suppose the nef cone $\mathrm{NE}(X/U)$ is generated by $\mathrm{NE}(X/U)_{K_X+S+B{\geq}0}$ and countably many extremal rays which do not accumulate in $\mathrm{NE}(X/U)_{K_X+S+B<0}$.  Then the contraction $f\colon X\to Z$ of $\Sigma$ exists so that $f$ is a projective morphism with $\rho(X/Z)=1$. 
\end{proposition}
\begin{proof} Perturbing the coefficients, we may assume that $(X,S+B)$ is plt. We may pick an ample over $U$ $\Q$-divisor $H$ such that $L=K_X+S+B+H$ is nef over $U$,  and $L^\perp$ is spanned by $\Sigma$. Let $A$ be an ample $\Q$-divisor  such that $(S+A)\cdot \Sigma =0$, then for any $0<\epsilon\ll 1$ we may pick an ample $\Q$-divisor $H_\epsilon \sim _\Q H+ \epsilon (S+A)$ such that $L_\epsilon =K_X+S+B+H_\epsilon$ is nef over $U$,  $(L_\epsilon)^\perp=\mathbb R [\Sigma]$, and $\mathbb E(L_\epsilon)\subset S$. To verify the last inclusion, notice that if  $V\subset X$ is a subvariety not contained in $S$, then $L_\epsilon |_V=(L+\epsilon (S+A))|_V$ is nef and big over $U$, and so $ L_{\epsilon}^{\dim V}\cdot V>0$. Replacing $L$ by $L_\epsilon$, we may assume that $\mathbb E(L)\subset S$.

By adjunction, 
$(S^\nu,B_{S^\nu})$ is klt where $K_{S^\nu}+B_{S^\nu} = (K_X+S+B)|_{S^{\nu}}$, the map $S^\nu \to S$ is the normalisation, and $H_{S^{\nu}} = H|_{S^{\nu}}$. By 
\cite[Theorem 1.4(1)]{hnt} and \cite[Theorem G]{BMPSTWW20}, 
 $L|_{S^\nu}$ is semiample over $U$ (note that in the mixed characteristic case, by \cite[Theorem 1.4(1)]{hnt}, we may assume that $\dim \pi(S) \geq 1$, and so \cite[Theorem G]{BMPSTWW20} applies). Since $S$ is normal up to universal morphism, 
 $L|_S$ is semiample by \cite[Lemma 1.4]{Keel99} and Theorem \ref{t-semiampless-universal-homeo}. Since $\mathbb{E}(L) \subseteq S$, \cite[Theorem 0.2]{Keel99} and \cite[Theorem 1.2]{Witaszek2020KeelsTheorem} imply that $L$ is semiample over $U$, and thus it induces a projective morphism $f\colon X\to Z$ contracting $\Sigma$ and such that $f_*\OO _X=\OO _Z$ and $\rho (X/Z)=1$.
\end{proof}
\begin{theorem}\label{t-mmp}Let $(Y,\Delta )$ be a four-dimensional $\Q$-factorial dlt pair with standard coefficients defined over 
a perfect field
of characteristic $p> 5$ (or a DVR of characteristic $(0,p)$ for $p>5$ with perfect  residue field). Assume that there exists a projective birational morphism $\pi \colon  Y \to X$ over a normal $\Q$-factorial variety $X$ such that ${\rm Ex}(\pi ) \subset \lfloor \Delta \rfloor$.  Then $\pi$-relative divisorial and flipping contractions and flips exist for $K_Y + \Delta$, and we can run a $(K_Y + \Delta )$-mmp over $X$ which terminates with a minimal model $\phi \colon Y \dasharrow Y'$.

\end{theorem}
\begin{proof} If $K_Y+\Delta$ is $\pi$-nef, then we are done. Otherwise we run a $(K_Y+\Delta)$-mmp over $X$. Let $\phi _k \colon Y=Y_0\dasharrow Y_1 \dasharrow \ldots \dasharrow Y_k$ be a sequence of flips and contractions and let $\Delta_k$ be the strict transforms of $\Delta$ on $Y_k$.
 Since $X$ is $\Q$-factorial, there exists an effective exceptional divisor $E$ on $Y_k$ such that $-E$ is ample over $X$. In particular $\lfloor \Delta _k\rfloor$ contains all $\pi_k\colon Y_k\to X$ exceptional curves.  We must show that we can continue the mmp. The termination will then follow by Theorem \ref{t-st4} as both the flipping and flipped locus must be contained in $\mathrm{Ex}(\pi)\subseteq \lfloor \Delta \rfloor$. Note that if $K_{Y_k}+\Delta_k $ is nef over $X$, then $Y_k$ is the required minimal model. Therefore, we may assume that $K_{Y_k}+\Delta_k $ is not  nef over $X$.

We start by establishing the cone theorem.
Let $S_i$ be all the $\pi_k$-exceptional divisors for $1 \leq i \leq r$. 
Restricting to the normalisation $S_i^\nu$ of $S_i$, we get that $(S_i^{\nu}, \Delta_{S_i^{\nu}})$ is dlt, where $K_{S_i^\nu}+\Delta _{S_i^\nu}=(K_{Y_k}+\Delta_k)| _{S_i^\nu}$. 
The three-dimensional cone theorem (\cite[Theorem 1.1]{DW19} and \cite[Theorem H]{BMPSTWW20}) states that
\[ \mathrm{NE}(S_i^\nu/X)=\mathrm{NE}(S_i^\nu/X)_{K_{S_i^\nu}+\Delta_{S_i^\nu}  \geq 0}+\sum _{j\geq 1}\mathbb R _{\geq 0}[\overline{\Gamma}_{i,j}],\]
where $0 > (K_{S_i^\nu}+\Delta_{S_i^\nu})\cdot \overline{\Gamma}_{i,j}= (K_{Y_k}+\Delta_k)\cdot \Gamma_{i,j}.$ Here $\{\overline{\Gamma}_{i,j}\}_{j\geq 1}$ is a countable collection of curves on $S_i^\nu$ and $\Gamma _{i,j}$ denote their images on $S_i \subseteq Y_k$.
Since ${\rm Ex }(\pi_k ) \subset {\rm Supp }(E)$,
there is a surjection $\sum_{i=1}^r \mathrm{NE}(S_i^\nu/X)\to \mathrm{NE}(Y_k/X)$
and hence 
\[
\mathrm{NE}(Y_k/X)=\mathrm{NE}(Y_k/X)_{K_{Y_k}+\Delta_k \geq 0}+\sum _{1 \leq i \leq r, 1 \leq j}\mathbb R _{\geq 0}[\Gamma_{i,j}].
\]
Note that for any ample $\Q$-divisor $H_k$ on $Y_k$, the set $\{ \Gamma _{i,j} \mid (K_{Y_k}+\Delta _k+H_k)\cdot \Gamma _{i,j}\leq 0\}$ is finite, and so extremal rays do not accumulate in $\mathrm{NE}(Y_k/X)_{K_{Y_k}+\Delta_k < 0}$.

Now, pick an extremal curve $C = \Gamma_{i,j}$ for some $i,j$. Then, there is an effective $\pi_k$-exceptional divisor $S$ such that $S \cdot C < 0$, and hence a contraction $f \colon Y_k \to Z$ of $C$ exists by Proposition \ref{p-cont} applied to $(Y_k,\Delta_k)$.
If $f$ is a divisorial contraction, we let $Y_{k+1}=Z$. 
If $f$ is a flipping contraction, then by \cite[Lemma 3.1]{HW19a} (this lemma is stated over a field but it holds in much wider generality such as over DVRs) there exists an $f$-ample irreducible divisor $A$ which is $\pi$-exceptional and hence contained in $\lfloor \Delta_k \rfloor$.
By Theorem \ref{t-sflip}, the flip $f^+\colon Y^+\to Z$ exists. We let $Y_{k+1}=Y^+$, $\psi \colon Y_k\dasharrow Y_{k+1}$, and $K_{Y_{k+1}}+\Delta _{k+1}=\psi _*(K_{Y_{k}}+\Delta _{k})$. \qedhere

\end{proof}

\begin{proof}[Proof of Theorem \ref{t-mmp+}] 
Note that the proof of Theorem \ref{t-mmp} implies the existence of all relevant divisorial and flipping contractions and the termination of the relevant flips. It suffices therefore to show the existence of flips. 
Suppose that $f\colon Y\to Z$ is a flipping contraction over $X$. We must show that the corresponding flip $Y^+\to Z$ exists.
Let $\zeta (\Delta )$ be the number of components of $\Delta$ whose coefficient is not contained in the standard set $\{ 1-\frac 1 m \mid m\in \mathbb N \}\cup \{ 1\}$. We will prove the result by induction on $\zeta (\Delta )$. 



Note that if $\zeta (\Delta )=0$, then  by Theorem \ref{t-mmp} the required flip exists.
Therefore, we assume that $\zeta (\Delta )>0$, and so we may write $\Delta =aS+B$ where $a\not \in \{ 1-\frac 1 m \mid m\in \mathbb N \}\cup \{ 1\}$. Note that $S$ is not exceptional over $X$ (as the exceptional divisors occur with coefficient one). 
Let $\nu \colon Y'\to Y$ be a log resolution which is an isomorphism at the generic points of strata of $\lfloor \Delta \rfloor$ and let $B'=\nu ^{-1}_*B+{\rm Ex}(\nu)$, $S'=\nu ^{-1}_* S$.
Since $\zeta (S'+B')<\zeta (\Delta)$, we may run the $(K_{Y'}+S'+B')$-mmp over $Z$. We obtain a birational contraction $\phi\colon Y'\dasharrow Y''$ such that $K_{Y''}+S''+B''=\phi _*(K_{Y'}+S'+B')$ is nef over $Z$ and all components of $\lfloor S''+B''\rfloor$ are normal up to universal homeomorphism. We now run a $(K_{Y''}+aS''+B'')$-mmp over $Z$ with scaling of $(1-a)S''$. Note that
this is also a $(K_{Y''}+B'')$-mmp and $\zeta (B'')<\zeta (\Delta )$. Therefore the required minimal model $Y''\dasharrow Y^+$ exists. Since the log resolution was an isomorphism at the generic points of strata of $\lfloor \Delta \rfloor$, it is easy to see that this is the required flip (cf.\ \cite[Section 7.1]{HW19b}).
\end{proof}

\subsection{Applications}

We begin by proving that   fourfold $\Q$-factorial klt singularities in positive characteristic are $W\OO$-rational.
\begin{corollary} \label{c-witt-rational} Let $X$ be a $\Q$-factorial klt four-dimensional variety over a perfect field $k$ of characteristic $p>5$ admitting a log resolution of singularities. Then $X$ has $W\mathcal O$-rational singularities. 
\end{corollary}
\begin{proof} 
Let $\pi \colon Y \to X$ be a log resolution and run the $(K_Y+{\rm Ex}(\pi))$-mmp over $X$. By standard arguments, this mmp contracts ${\rm Ex}(\pi)$ and so its output $Y'\to X$ is a small birational morphism and hence an isomorphism as $X$ is $\Q$-factorial. We shall show that if $g \colon Y \dashrightarrow Y_n$ is a sequence of steps of the mmp, with the induced map $\pi_n \colon Y_n \to X$, then $R^i\pi_*W\OO_{Y,\Q} = R^i(\pi_n)_*W\OO_{Y_n,\Q}$. As a consequence $R^i\pi_*W\OO_{Y,\Q} = 0$ for $i>0$, and so $X$ has $W\OO$-rational singularities. Thus, inductively, after replacing $Y$ by some steps of the mmp, we may assume that $(Y, {\rm Ex}(f))$ is dlt. Let $f \colon Y \to Z$ be a Mori contraction over $X$ with $\rho(Y/Z)=1$. 

Assume that $f$ is divisorial. There exists an $f$-anti-ample irreducible divisor $S \subseteq {\rm Ex}(f)$ with the induced map $f_S \colon S \to X$. Set $K_{S^{\nu}}+\Delta_{S^{\nu}} = (K_Y + {\rm Ex}(f))|_{S^{\nu}}$, where $S^{\nu}$ is the normalisation of $S$ with the induced map $f_{S^{\nu}} \colon S^{\nu} \to X$. Then $(S^{\nu},\Delta_{S^{\nu}})$ is dlt by Lemma \ref{l-dltad} and $-(K_{S^{\nu}}+\Delta_{S^{\nu}})$ is $f_{S^{\nu}}$-ample. Up to perturbation, we can assume that $(S^{\nu}, \Delta_{S^{\nu}})$ is klt, and so, by \cite[Theorem 3.11]{NT20}, 
\[
R^i(f_{S^{\nu}})_*W\OO_{S^{\nu},\Q}=0
\]
for $i>0$. Since $S$ is normal up to universal homeomorphism (Lemma \ref{c-nuuh}), \cite[Lemma 2.17]{NT20} and \cite[Lemma 2.20 and 2.21]{GNT16} together with the Leray spectral sequence imply that $R^i(f_{S})_*W\OO_{S,\Q}=0$ for $i>0$. Let $u \colon S \to Y$ be the inclusion and let $I_S$ be the ideal sheaf defining $S$. Then, we get the following short exact sequence (cf.\ \cite[Proof of Proposition 3.4]{GNT16})
\[
0 \to WI_{S,\Q} \to W\OO_{Y,\Q} \to u_*W\OO_{S,\Q} \to 0.
\]
By \cite[Proposition 2.23]{GNT16}, $R^if_* WI_{S,\Q}=0$ for $i>0$, and so 
\[
R^if_*W\OO_{Y,\Q}=R^if_*(u_*W\OO_{S,\Q}) = R^i(f_S)_*W\OO_{S,\Q} = 0,
\]
where the second equality follows from \cite[Lemma 2.20 and Lemma 2.21]{GNT16}. Moreover, $f_*W\OO_{Y,\Q} = W\OO_{Z,\Q}$ by \cite[Lemma 2.17]{NT20}, and so we can conclude the proof when $f$ is divisorial by the Leray spectral sequence.

Suppose that $f$ is a flipping contraction and let $\pi_Z \colon Z \to X$ be the induced map. Let $\phi \colon Y \dashrightarrow Y^+$ be the flip, let $f^+ \colon Y^+ \to Z$ be the flipped contraction, and let $\pi^+ \colon Y^+ \to X$ be the induced map. As before, by \cite[Lemma 3.1]{HW19a} and Theorem \ref{t-sflip}, both $Y$ and $Y^+$ are relatively F-regular over $Z$. In particular, $Rf_* \OO_Y = \OO_Z$ and $Rf^+_* \OO_{Y^+} = \OO_Z$ (see \cite[Theorem 6.8]{schwedesmith10} and the proof of \cite[Proposition 3.4]{HW19b}). By \cite[Lemma 2.17]{NT20} and \cite[Lemma 2.19 and Lemma 2.21]{GNT16},  $f_* W\OO_{Y,\Q} = f^+_* W\OO_{Y^+,\Q} = W\OO_{Z,\Q}$ and $R^if_* W\OO_{Y,\Q} = R^if^+_* W\OO_{Y^+,\Q} = 0$ for $i>0$. By the Leray spectral sequence:
\[
R^i\pi_*W\OO_{Y,\Q} = R^i(\pi_Z)_*W\OO_{Z,\Q} = R^i\pi_*^+W\OO_{Y^+,\Q}. \qedhere 
\]



\end{proof}

Next, we show the existence of dlt modifications.
\begin{corollary}\label{c-dlt-mod} Let $(X,\Delta)$ be a four-dimensional $\Q$-factorial log pair with standard coefficients admitting a log resolution and defined over a perfect field of characteristic $p>5$ or a DVR of characteristic $(0,p)$ for $p>5$ with perfect residue field. Then a dlt modification of $(X,\Delta)$ exists, that is, a birational morphism $\pi \colon Y \to X$ such that $(Y, \pi^{-1}_*\Delta + {\rm Ex}(\pi))$ is dlt, $\Q$-factorial, and minimal over $X$.
\end{corollary} 
\begin{proof}
Let $\pi \colon Y \to X$ be a log resolution of $(X, \Delta)$. Then, a dlt modification is a minimal model of $(Y,\pi^{-1}_*\Delta + {\rm Ex}(\pi))$ over $X$ (see Theorem \ref{t-mmp+}).
\end{proof}

Finally, we prove that inversion of adjunction holds.
\begin{corollary}\label{c-inversion-of-adjunction} Consider a $\Q$-factorial four-dimensional log pair $(X,S+B)$ with standard coefficients defined over a perfect field of characteristic $p>5$ or a DVR of characteristic $(0,p)$ for $p>5$ with perfect residue field, where $S$ is a prime divisor with no common component with $B\geq 0$. Assume that $(X,S+B)$ admits a log resolution. 

Then $(X,S+B)$ is plt on a neighbourhood of $S$ if and only if $(\tilde S,B_{\tilde S})$ is klt, where $\tilde S$ is the normalisation of $S$ and $B_{\tilde S}$ is the different.
\end{corollary}
\begin{proof}
By considering a log resolution of $(X,S+B)$ it is easy to see that if $(X,S+B)$ is plt, then $(\tilde S,B_{\tilde S})$ is klt. Thus, we can assume that $(\tilde S,B_{\tilde S})$ is klt and aim to show that $(X,S+B)$ is plt near $S$. 

Let $\pi \colon Y \to X$ be a dlt modification of $(X,S+B)$ (see Corollary \ref{c-dlt-mod}) and write $K_Y + S_Y + B_Y = \pi^*(K_X+S+B)$. By definition (of a dlt modification) for any $\pi$-exceptional irreducible divisor $E$ we have that $E \subseteq \lfloor B_Y \rfloor$.  Write
\[
(\pi|_{S_Y})^*(K_{\tilde S} + B_{\tilde S}) = (K_Y+S_Y+B_Y)|_{\tilde S_Y} = K_{\tilde S_Y} + B_{\tilde S_Y}, 
\]
where $\tilde S_Y \to S_Y$ is the normalisation of $S_Y$, and $B_{\tilde S_Y}$ is the different. Let $E$ be a $\pi$-exceptional divisor intersecting $S_Y$. Since $E \subseteq \lfloor B_Y \rfloor$ and $(Y,S_Y+{\rm Ex}(\pi))$ is dlt, we must have that $E \cap S_Y \subseteq \lfloor B_{\tilde S_Y} \rfloor$. Since $E \cap S_Y \neq \emptyset$, this contradicts $(\tilde S, B_{\tilde S})$ being klt. 

Therefore we may assume that $E\cap S_Y=\emptyset$, so that $Y=X$ near $S$ and hence $(X,S+B)$ is dlt on a neighbourhood of $S$. Since $S$ is irreducible, $(X,S+B)$ is in fact plt.
\end{proof} 

\section{Relative mmp}
Throughout this section, we assume that log resolutions of all log pairs with the underlying variety being birational to $X$ as below exist (and are given by a sequence of blow-ups along the non-snc locus). The assumption on the field (in the positive characteristic case) being algebraically closed is necessary to invoke Bertini's theorem for strongly F-regular singularities.

We start by proving the following base point free theorem. This result would follow easily from \cite{CT17} and \cite{witaszek2021relative} if we knew abundance for slc threefolds.

\begin{proposition} \label{p-bpf-families2}
Let $(X,\Delta)$ be a four-dimensional $\Q$-factorial dlt pair with standard coefficients projective over a DVR $R$ with perfect  residue field of characteristic $p>5$. Let $s, \eta \in \Spec R$ be the special and the generic point, respectively, and let $\phi \colon X \to \Spec R$ be the natural projection. When $R$ is purely positive-characteristic, we also assume that it is a local ring of a curve $C$ defined over an algebraically closed field and that $(X,\Delta) := (\mX, \Phi) \times_C \Spec(R)$ for a $\Q$-factorial four dimensional dlt pair $(\mX, \Phi)$ which is projective over $C$.   

Suppose that  $K_X+\Delta$ is nef and big, and $\lfloor \Delta \rfloor = {\rm Supp}(\phi^{-1}(s))$. Then $K_X+\Delta$ is semiample.
\end{proposition}
\begin{proof}


Write $\Supp\, X_{s}=\sum _{i=1}^r E_i$ for irreducible divisors $E_i$, and $L := K_X+\Delta$. Since $(X, \Delta)$ is klt over $\eta$, we get that $L|_{X_{\eta}}$ is semiample by \cite[Theorem 1.4]{DW19} or the base point free theorem in characteristic zero \cite{KM98}. By \cite[Theorem 1.1]{CT17} and \cite[Theorem 1.2]{witaszek2021relative}, it thus suffices to show that $L|_{X_s}$ is semiample.

Since $L$ is big, we may assume that  $\Delta + L \sim _{\Q}H+F+G$ where  $H$ is ample, $F\geq 0$  is supported on $X_s$, and the  support of $G \geq 0$ contains no divisors of $X_s$. 
In this paragraph we reduce the proposition to the case when $G$ does not contain any log canonical centres of $(X,\Delta)$. To this end, let $\pi\colon Y\to X$ be a dlt modification of $(X, \Delta +\delta G)$ for some $0<\delta \ll 1$ (see  Corollary \ref{c-dlt-mod}). 
Then we may assume that $K_Y+\Delta_Y=\pi ^*(K_X+\Delta)$ where $\Delta _Y=\pi ^{-1}_*\Delta +{\rm Ex}(\pi)$ and $\pi ^{-1}_* G$ contains no strata of $\lfloor \Delta_Y \rfloor$; indeed as $(X,\Delta )$ is dlt, the divisorial non-klt places of $(X,\Delta +\delta G)$  and $(X,\Delta)$  coincide on a given resolution (on which we run the mmp to construct the dlt modification) for $0 < \delta \ll 1$. Let $P$ be an effective $\pi$-exceptional divisor such that $-P$ is $\pi$-ample. Then the support of $P$ is contained in $Y_s$ and   
we may assume that $\pi ^*H-P$ is ample over $\Spec R$.  Note that $\Delta _Y-\pi ^* \Delta$ 
is supported on ${\rm Ex}(\pi)  \subseteq Y_s$.  We have \begin{align*}
\Delta _Y+\pi ^*L &= \Delta_Y-\pi^*\Delta + \pi^*(H+F+G)\\
&= (\pi^*H-P)+(P+\Delta _Y-\pi ^* \Delta +\pi ^*F+\pi ^*G-\pi^{-1}_*G) + \pi^{-1}_*G\\
&= H_Y+F_Y+G_Y,   
\end{align*}
where $G_Y=\pi ^{-1}_*G$, $H_Y=\pi ^*H-P  - aY_s$ is ample over $\Spec R$, and $F_Y=P+(\Delta _Y-\pi ^* \Delta) +\pi ^*F+(\pi ^*G-\pi^{-1}_*G) + aY_s$ is supported on $Y_s$ (we choose $a \gg 0$ so that $F_Y \geq 0$). Replacing $X, \Delta,L, H,F,G$ by $Y,\Delta _Y,\pi ^*L, H_Y, F_Y, G_Y$, we may assume that $(X,\Delta +\delta G)$ is dlt, $\lfloor \Delta \rfloor = \Supp\, X_s$, $F$ is supported on $X_s$, and $H$ is ample. 


Since $H$ is ample, we may further assume that the support of $F=\sum f_iF_i$ equals $X_s$ where the $f_i$ are chosen generically (here $F_i$ are distinct prime divisor). Moreover, we have that 
\[
(X,\Delta_{\epsilon,\delta}:=(1-\delta)\Delta {+\epsilon X_s}+\delta (H+F+G))
\] 
is klt for $0 < \delta \ll 1$ and some possibly negative $\epsilon$ such that $|\epsilon |\ll 1$. Note that $K_X + \Delta_{\epsilon,\delta} \sim_{\Q} (1+\delta)L$. 

Fixing $\delta$ and increasing $\epsilon$, since the $f_i$ are chosen generically, we obtain a sequence of rational numbers $\epsilon<\epsilon _1<\epsilon _2<\ldots <\epsilon _r$ such that  $\lfloor \Delta_{\epsilon _i,\delta}\rfloor =U_i:=\sum _{j=1}^iE_j$ {and $E_i$ occurs with coefficient one in $\Delta_{\epsilon _i,\delta}$.}
Here of course we have re-indexed the $E_i$ accordingly. Note that all $E_i$ are defined over a perfect field of positive characteristic. We claim  that 
\begin{claim}\label{c-sa} $(K_X + \Delta)|_{E_i^{\nu}}$ is semiample  where $E_i^{\nu}\to E_i$ is the normalisation, and hence also $(K_X + \Delta)|_{E_i}$ is semiample as the $E_i$ are normal up to universal homeomorphism (see Theorem \ref{t-semiampless-universal-homeo}). \end{claim}


Granting the claim and proceeding by induction, we may assume that
 $(K_X+\Delta)|_{U_{i-1}}$ is semiample and we must show that $(K_X+\Delta)|_{U_i}$ is semiample. By \cite[Corollary 2.9]{Keel99}, it suffices to show that $g_2|_{U_{i-1}\cap E_i}$ has connected geometric fibres where $g_2 \colon E_i^\nu \to  V$ is the morphism associated to the semiample $\Q$-divisor $(K_X+\Delta)|_{E_i^\nu}$.
 Note that $(K_X+\Delta)|_{E_i^\nu}\equiv _{V}0$, hence $-(K_{E_i^\nu}+\Delta'_{E_i^\nu}):=-(K_X+\Delta_{\epsilon_i,\delta}-\delta H)|_{E_i^\nu }$ is 
 {ample} over $V$. By \cite[Theorem 1.2]{NT20} (this requires the base field to be perfect), the fibres of the non-klt locus of $({E_i^\nu},\Delta'_{E_i^\nu})$ are geometrically connected. Since this non-klt locus coincides with $U_{i-1}\cap {E_i}$ {(as $K_X+\Delta +\delta(H+G)$ is dlt)}, the statement of the proposition follows.

 \begin{proof}[Proof of Claim \ref{c-sa}] We proceed by induction. We warn the reader that $E^{\nu}_i$ need not be $\Q$-factorial. Set $K_{E_i^\nu}+\Delta_{E_i^\nu}=(K_X+\Delta)|_{E_i^\nu}$. 

 First, if $L|_{E^{\nu}_i}$ is big, then it is semiample by \cite[Theorem 1.1]{waldronlc} and base change to an algebraically closed field. Hence, we can assume that it is not big. In particular, $L|_{E^{\nu}_i}-\gamma \lfloor \Delta_{E_i^\nu}\rfloor$ is not pseudo-effective for any $\gamma > 0$ (note that $\lfloor \Delta_{E_i^\nu}\rfloor = \lfloor \Delta - E_i \rfloor|_{E_i^\nu}$ is $\Q$-Cartier). Indeed, assume otherwise and write $L \sim_{\Q} H'+F'+G'$, where $H'$ is ample, $F'\geq 0$ is supported on $X_s$, and the support of $G'\geq 0$ contains no divisors of $X_s$. By shifting $F'$ by a multiple of $\phi^{-1}(s)$, we may assume that $F'$ does not contain $E_i$ in its support (but it is not necessarily effective any more). Thus we can write
 \begin{align*}
 (m+1)L|_{E^{\nu}_i} &\sim_{\Q} H'|_{E^{\nu}_i} + mL|_{E^{\nu}_i} + F'|_{E^{\nu}_i} + G'|_{E^{\nu}_i} \\
 &\sim_{\Q} H'|_{E^{\nu}_i} + (mL|_{E^{\nu}_i}-t\lfloor \Delta_{E_i^\nu}\rfloor) + (F'|_{E^{\nu}_i}+t\lfloor \Delta_{E_i^\nu}\rfloor) + G'|_{E^{\nu}_i}
 \end{align*}
 for $m,t \in \mathbb{N}$ 
 such that 
 $mL|_{E^{\nu}_i}-t\lfloor \Delta_{E_i^\nu}\rfloor$ and $F'|_{E^{\nu}_i}+t\lfloor \Delta_{E_i^\nu}\rfloor$ are pseudo-effective. This shows that $L|_{E^{\nu}_i}$ is big.

Let $(\overline{E}_i, \Delta_{\overline{E}_i})$ be a dlt modification of $(E^{\nu}_i, \Delta_{E^{\nu}_i})$ so that $\overline{E}_i$ is $\Q$-factorial (see \cite[Corollary 3.6]{hnt}).
Let $r$ be the Cartier index of $K_{\overline{E}_i} + \Delta_{\overline{E}_i}$ and let $L_i$ be the pullback of $L|_{E_i^{\nu}}$. Note that $L_i = K_{\overline{E}_i} + \Delta_{\overline{E}_i}$.
We will run the $(K_{\overline{E}_i}+\Delta_{\overline{E}_i}-\gamma \lfloor \Delta_{\overline{E}_i}\rfloor)$-mmp 
  for $\gamma = \frac{1}{6r+1}$, which by the above must terminate with a Mori fibre space (cf.\ \cite[Theorem 2.13]{GNT16}). 
 
 We claim that each step of the mmp is $(K_{\overline{E}_i}+\Delta_{\overline{E}_i})$-trivial. We start by showing this for the first step. Note that 
 $(K_{\overline{E}_i}+\Delta_{\overline{E}_i}-\lfloor \Delta_{\overline{E}_i}\rfloor)\cdot \Gamma \geq -6$ for any $(K_{\overline{E}_i}+\Delta_{\overline{E}_i}-\gamma\lfloor \Delta_{\overline{E}_i}\rfloor)$-negative extremal ray $\Gamma$ (see \cite[Theorem 1.1]{DW19}; observe that the parameters $d$ therein are equal to one as the field is perfect).
 But  if $(K_{\overline{E}_i}+\Delta_{\overline{E}_i})\cdot \Gamma >0$, then  $(K_{\overline{E}_i}+\Delta_{\overline{E}_i})\cdot \Gamma \geq 1/r$ by the definition of $r$, 
 and hence 
 \[
 (K_{\overline{E}_i}+\Delta_{\overline{E}_i}-\gamma \lfloor \Delta_{\overline{E}_i}\rfloor)\cdot \Gamma \geq \frac {1-\gamma}{r}-6\gamma=0
 \]
 which is a contradiction. In particular, $(K_{\overline{E}_i}+\Delta_{\overline{E}_i})\cdot \Gamma =0$. Hence, the nefness of $K_{\overline{E}_i}+\Delta_{\overline{E}_i}$ is preserved by the first step of the mmp, and so is its Cartier index $r$ by \cite[Corollary 1.5]{ABL20} (precisely, if $\overline{E}_i \to Z$ is the contraction of $\Gamma$ given by $L' = K_{\overline{E}_i}+\Delta_{\overline{E}_i}-\gamma\lfloor \Delta_{\overline{E}_i}\rfloor+A$ for some ample $\Q$-divisor $A$, then we apply this result to $rL_i \sim_{\Q,Z} K_{\overline{E}_i}+\Delta_{\overline{E}_i}-\gamma\lfloor \Delta_{\overline{E}_i}\rfloor+A + (rL_i - L')$ over $Z$). Hence, we may repeat this procedure for all steps of the $(K_{\overline{E}_i}+\Delta_{\overline{E}_i}-\gamma \lfloor \Delta_{\overline{E}_i}\rfloor)$-mmp.

 
 

 By replacing $\overline{E}_i$ by the output of this mmp and $\Delta_{\overline{E}_i}$ together with $L_i$ by their pushforwards, we can assume that we have a $(K_{\overline{E}_i}+\Delta_{\overline{E}_i}-\gamma \lfloor \Delta_{\overline{E}_i}\rfloor)$-Mori fibre space $\psi \colon \overline{E}_i\to Z$ (in particular $Z$ is normal). Since $\psi$ is $L_i$-trivial, $\lfloor \Delta_{\overline{E}_i}\rfloor$ is relatively ample so that one of its components, say $W$, dominates $Z$.
 We claim that $L_i|_{\lfloor \Delta_{\overline{E}_i}\rfloor}$ (and so $L_i|_W$) is semiample.
 This can be checked after a base change to an algebraically closed field and on a dlt modification of $(\overline{E}_i, \Delta_{\overline{E}_i})$ (cf.\ \cite[Lemma 2.11(3)]{CT17})  in which case it follows from  \cite[Theorem 1.3]{waldronlc}.
 Now $L_i \sim_{\Q} \psi ^* D$ for some $\Q$-divisor $D$ on $Z$. But then $(\psi |_W)^*D \sim_{\Q} L_i|_W$, which is semiample, and hence both $D$ and $L_i$ are semiample as $Z$ is normal (cf.\ \cite[Lemma 2.11(4)]{CT17}). \qedhere
 
\end{proof}
\end{proof}

\begin{proposition} \label{p-bpf-families} With notation as in the first paragraph of Proposition \ref{p-bpf-families2}, suppose that  ${\rm Supp}(\phi^{-1}(s))\subset  \lfloor \Delta \rfloor$. Let $L$ be a nef and big $\Q$-Cartier $\Q$-divisor on $X$ such that $L-(K_X+\Delta)$ is ample. Then $L$ is semiample and induces a morphism $f\colon X \to Z$ over $\Spec R$. In particular, every $f$-numerically trivial $\Q$-Cartier divisor descends to a $\Q$-Cartier divisor on $Z$.
\end{proposition}
\begin{proof}
This follows from Proposition \ref{p-bpf-families2} by perturbation.
\end{proof}

\begin{proof}[Proof of Theorem \ref{t-mmp-families+}] The proof is similar to that of \cite[Proposition 4.1]{HW19a}. We will run an arbitrary $(K_{X}+\Delta)$-mmp.

In the positive characteristic case, we will implicitly work with the spread-out $\mathcal{X}$ of $X$ and repeatedly replace $C$ by an appropriate neighbourhood of the special point $s\in U\subset C$ and $\mathcal X$ by $\mathcal X\times _C U$, so that we can apply results on the mmp for finite type schemes over a field. Let $\eta$ be the generic point of $\Spec R$.  \\

Following the proof of Theorem \ref{t-mmp}, we will run a $(K_X+\Delta)$-mmp  $X =X_0\dasharrow  X_1\dasharrow \ldots \dasharrow  X_k$. We will denote the central fibre by $ X_{k,s}$ and the pushforward of $\Delta$ on $X_k$ by $\Delta_k$. Note that $X_{k,s}\sim _{\Spec R} 0$.
By the same argument as that in Theorem \ref{t-mmp}, the cone theorem is valid for $K_{ X_k}+\Delta _k$. Hence if $K_{ X_k}+\Delta _k$ is not nef over $\Spec R$, then there exists  a $(K_{ X_k}+\Delta _k)$-negative extremal ray over $\Spec R$ spanned by a curve $\Sigma \subset X _{k,s}$.

Let $G$ be an ample $\Q$-divisor such that $L = K_{ X_k} + \Delta_k + G$ is nef and $L^{\perp} = \mathbb R[\Sigma]$. Then $L$ is semiample by Proposition \ref{p-bpf-families}.
Let $f\colon X _k\to Z$ be the corresponding contraction, which is birational as $\kappa(K_X+\Delta\,/\,{\rm Spec}(R))\geq 0$.  There are two cases.
If  $f\colon X _k\to Z$ is a divisorial contraction  (over $\Spec R$), then we may set $X _{k+1}:= Z$ and continue our mmp  over $\Spec R$. Thus we must show that if the induced morphism $f\colon X _k\to Z$ (over $\Spec R$) is a flipping contraction, then the flip $f^+ \colon X_{k+1}\to Z$ exists. Further, we must show that there is no infinite sequence of flips. The latter follows from Theorem \ref{t-st4}. As for the former, by means of perturbation, we may assume that $\lfloor \Delta_k \rfloor = \Supp\,  X_{k,s}$, and consider two cases. 

\begin{claim}\label{c-central} If the flipping locus is contained in the special fibre $ X_{k,s}$, then the flip $ X_{k}\dasharrow X_{k+1}$ exists.
\end{claim}
\begin{proof}
Let $\Sigma$ be a flipping curve.  If $\Sigma \cdot S\ne 0$ for some component $S$ of $ X _{k,s}$, then since $\Sigma\cdot X _{k,s}=0$, we may assume that there is another component $E$ of $X _{k,s}$ such that $\Sigma\cdot S<0$ and $\Sigma\cdot E>0$ (up to swapping $S$ and $E$). Thus the flip exists by Theorem \ref{t-sflip}. We may therefore assume that  $\Sigma\cdot S=0$ for every component of $\mathcal X _{k,s}$. We follow the proof of \cite[Proposition 4.1]{HW19a}. Now let $S$ be a component of $X_{k,s}$ such that $\Sigma \subseteq S$.  Since $S\cdot \Sigma=0$, we may assume that $S':=f_*S$ is also $\Q$-Cartier. Let $H'$ be a reduced $\Q$-Cartier divisor on $Z$ such that
\begin{enumerate}
\item $H=f^{-1}_*H'$ contains  ${\rm Ex}(f)$,
\item  $(X_k, \Delta_k + H)$ is dlt over the generic point $\eta \in \Spec R$,
\item for any proper birational morphism $h\colon Y\to Z$ such that $Y$ is $\Q$-factorial, we have that $N^1(Y/Z)$ is generated by the irreducible components of the strict transform of $H'$ and the $h$-exceptional divisors, and
\item $H$ and $\Delta _k$ have no common components. 
\end{enumerate}
Explicitly, we take $H_1$ and $H_2$ as in Lemma \ref{l-weird-bertini}. Then $H=H_1+H_2$ descends to a $\Q$-Cartier divisor $H'$ on $Z$ by Proposition \ref{p-bpf-families}. We leave the verification of (3) to the reader (here we use that $\rho(X_k/Z)=1$, otherwise we would need to enlarge $H'$).

Let $p\colon Y\to  X_k$ be a dlt modification of $(X_k,\Delta _k + H)$ (see Corollary \ref{c-dlt-mod}). Since the generic points of $H_1\cap H_2$ are simple normal crossing in $X_{k,\eta}$, we may assume that the dlt modification is constructed from a log resolution which is an isomorphism over the generic points of $(H_1 \cap H_2)|_{X_{k,\eta}}$. In particular, $p$ must be an isomorphism over the generic point $\eta \in R$. Set $h = f \circ p \colon Y \to Z$. 

First we will run a $(K_Y+\Delta _Y+H_Y)$-mmp over $Z$ where $H_Y=p^{-1}_* H$ and $\Delta _Y=p_*^{-1}\Delta_k +{\rm Ex}(p)$. All extremal rays $R$ are contained in the support of $h^*H'$ and some component of $h^*H'$ has a non-zero intersection number with $R$ (by condition (3) above). Since $h^*H'\cdot R=0$, there are components $E,E'$ of ${\Supp}\, h^*H'$ such that $E\cdot R<0$ and $E'\cdot R>0$. Since the support of $h^*H'$ is contained in the support of $\lfloor  \Delta _Y+H_Y\rfloor$, the necessary flips exist by Theorem \ref{t-sflip} and we may run the required minimal model program. Note that by Theorem  \ref{t-st4}, there is no infinite sequence of such flips and hence we may assume that, up to replacing $Y$ by the output of the mmp, $K_Y+\Delta _Y+H_Y$ is nef over $Z$. 
We now run a $(K_Y+\Delta _Y)$-mmp with scaling of $H_Y$ over $Z$. If $R$ is a corresponding $(K_Y+\Delta _Y)$-negative extremal ray, then $H_Y\cdot R>0$. Since 
\[ H_Y\equiv _h -\sum b_jE_j,\qquad b_j\geq 0\]
for some $h$-exceptional divisors $E_j$, it follows that $R\cdot E_j<0$ for some $j$. Since $E_j$ is contained in $Y_s$, it is contained in the support of $h^*Z_s$. Since $R\cdot h^*Z_s=0$, there is a component $E'$ of $\lfloor  \Delta _Y\rfloor$ such that $R\cdot E'>0$ and the necessary flip exists by Theorem \ref{t-sflip}.

By Theorem  \ref{t-st4}, there is no infinite sequence of such flips and hence we may assume that $K_Y+\Delta_Y$ is nef over $Z$. By Proposition \ref{p-bpf-families2} applied to $K_Y+ \Delta_Y+h^*A$ for a sufficiently ample divisor $A$ on $Z$, we get that $K_Y+\Delta_Y$ is semiample over $Z$, and thus by replacing $Y$ by the image of the associated semiample fibration, we can assume that $K_Y + \Delta_Y$ is relatively ample. We claim that $Y\to Z$ is small, and hence $X _{k+1}:=Y$ is the required flip (the canonical ring $R(K_{X_{k}} + \Delta_k) = R(K_Z + f_* \Delta_k) = R(K_Y + \Delta_Y)$ is finitely generated).
To this end, note that if $p \colon W\to X _{k}$ and $q \colon W\to Y$ is the normalisation of the graph of the rational map $ X _{k}\dasharrow Y$, then we may write 
\[
 p^*(K_{X _{k}}+\Delta_k)-q^*(K_Y+\Delta_Y) = E
 \] 
 where $p_*E=0$ (as $X _{k}\to Z$ is small) and $-E$ is relatively {ample} over $Z$. By the negativity lemma,  $E\geq 0$. On the other hand, by construction, if $F$ is a $p$-exceptional but not $q$-exceptional divisor on $W$, then by definition of $\Delta_Y$ we have
\[
{\rm mult}_F(E)={\rm mult}_F(p^*(K_{X _{k}}+\Delta _k))-{\rm mult}_F(q^*(K_Y+\Delta_Y))\leq 1-1=0,
\]
and so, in fact, ${\rm mult}_F(E) = 0$. This contradicts $-E$ being relatively ample over $Z$.\end{proof}



Let $f_\eta\colon (X_{k,\eta}, \Delta _{k,\eta})\to Z_\eta $ be the restriction of $f$ to $X_\eta$. By the above claim, we may assume that $f_\eta$ is not the identity. Since $f$ is a flipping contraction, it is easy to see that $f_\eta$ is a flipping contraction and in particular $X_{k,\eta}$ is $\Q$-factorial, $\rho (X_{k,\eta}/Z_\eta )=1$, $X_{k,\eta}\to Z_\eta $ is small, and $-(K_{X_{k,\eta} }+ \Delta _{k,\eta})$ is ample over  $Z_\eta $. By \cite[Theorem 1.3]{DW19} in positive characteristic (and \cite{bchm06} in characteristic zero), the corresponding flip $f_\eta ^+\colon X_{k,\eta}^+\to Z_\eta $ exists. Let $f' \colon X'_k \to Z$ be the closure of the projective morphism $f^+_\eta$. Note that the fibre of $X'_k$ over $s$ may be highly singular and in particular non-normal and not even $R_1$.

Let $\Delta'_k$ be a divisor on $X'_k$ given as the sum of the closure of $\Delta^+_{k,\eta}$ and the support of $X'_{k,s}$. Let $p \colon Y \to X'_k$ be a log resolution of $(X'_k, \Delta'_k)$ and set $\Delta' _Y=p_*^{-1}\Delta'_k +{\rm Ex}(p)$. We now run a $(K_Y+\Delta'_Y)$-mmp  over $X'_{k}$. We must show that the necessary flips exist and that the corresponding mmp terminates. Let $v \colon Y\to V$ be a flipping contraction.
Note that if the flipping locus is contained in the central fibre $Y_s$, then the flip exists by Claim \ref{c-central}. If the flipping locus dominates $\Spec R$, then as $X'_{k,\eta}$ is $\Q$-factorial, by \cite[Lemma 3.1]{HW19a} (the same argument works in mixed characteristic), there is a $v$-ample $p$-exceptional divisor $A_\eta\subset Y_\eta$. As $\rho(Y/V)=1$, its closure $A\subset Y$ is also relatively ample and is contained in the support of $\lfloor \Delta' _Y\rfloor$. Similarly, we also have an effective $p$-exceptional divisor $F_\eta\subset Y_\eta$ which is $p$-anti-ample. Therefore, there is a component $S_\eta$ of $F_\eta$ such that its closure $S\subset Y$ is $v$-anti-ample. Thus the flip exists by Theorem \ref{t-sflip}.  
By Theorem \ref{t-st4} the above sequence terminates with the required dlt model which we again denote by $(Y,\Delta'_Y)$.  Since $(X'_k, \Delta'_k)$ is klt and $\Q$-factorial over $\eta$, we can assume that $p$ is an isomorphism over $\eta$. Set $h = f' \circ p \colon Y \to Z$. 

We will run a $(K_Y+\Delta'_Y)$-mmp  over $Z$ where $(Y,\Delta' _Y=p_*^{-1}\Delta'_k +{\rm Ex}(p))$ is dlt. Contractions exist for the same reason as before (see Proposition \ref{p-bpf-families}). Since $K_{Y_{\eta}} + \Delta'_{Y_{\eta}}$ is ample over $Z_{\eta}$, the contracted locus is contained in the fibre over $s \in \Spec R$. In particular, the necessary flips exist by Claim \ref{c-central}. The mmp terminates by Theorem \ref{t-st4}.

Hence, we can assume that $K_Y+\Delta'_Y$ is nef over $Z$. {By Proposition \ref{p-bpf-families2},} $K_Y+\Delta'_Y$ is semiample over $\Spec R$.
Replacing $Y$ by the image of the associated semiample fibration, we can assume that $K_Y + \Delta'_Y$ is relatively ample. We claim that $Y\to Z$ is small, and hence $X _{k+1}:=Y$ is the required flip.
To this end, note that if $p \colon W\to X _{k}$ and $q \colon W\to Y$ is the normalisation of the graph of the rational map $X _{k}\dasharrow Y$, then we may write 
\[
 p^*(K_{X _{k}}+\Delta _k)-q^*(K_Y+\Delta'_Y) = E
 \] 
 where $p_*E=0$ (as $X _{k}\to Z$ is small) and $-E$ is relatively {ample} over $Z$. By the negativity lemma,  $E\geq 0$. On the other hand, by construction, if $F$ is $p$-exceptional but not $q$-exceptional, then by definition of $\Delta'_Y$ we have
\[
{\rm mult}_F(E)={\rm mult}_F(p^*(K_{X _{k}}+\Delta _k))-{\rm mult}_F(q^*(K_Y+\Delta'_Y))\leq 1-1=0,
\]
and so, in fact, ${\rm mult}_F(E) = 0$. This contradicts $-E$ being relatively ample over $Z$.

\end{proof}
In the proof we used the following technical lemma.
\begin{lemma} \label{l-weird-bertini} Let $C$ be a regular affine curve defined over an algebraically closed field $k$ of positive characteristic $p>0$. Let $s \in C$ be a closed point and let $\eta \in C$ be the generic point. Let $(X,S+B)$ be a $\Q$-factorial dlt pair of absolute dimension four and projective over the curve $C$ such that $S = \lfloor S + B \rfloor = \Supp\, X_s$ and $B$ has standard coefficients. Let $f \colon X \to Z$ be a small contraction with $\rho(X/Z)=1$ and exceptional locus ${\rm Ex}(f)$ contained in $X_s$. Suppose that every irreducible component $S_i$ of $S$ is $f$-numerically trivial.

 Then there exist $\Q$-Cartier Weil divisors $H_i$ for $i \in \{1,2\}$ such that $H_1$ and $H_2$ are $f$-ample and $f$-anti-ample, respectively, $H_1+H_2$ is $f$-numerically trivial, and $(X,S+B+H_1+H_2)$ is dlt over $\eta$.
 
 The same statement holds if $X$ is defined over $C=\Spec R$ for a DVR $R$ of mixed characteristic $(0,p>0)$.
\end{lemma}
\begin{proof}
Pick an $f$-ample Cartier divisor $D$ on $X$ and an ample Cartier divisor $A$ on $Z$. Set $D_i = (-1)^{i+1}D$ and $L_i := D_i + m_if^*A$ for $m_2 \gg m_1 \gg 0$ and $i \in \{1,2\}$. We claim that there exist effective divisors $H_{i,\eta} \in |L_i|_{X_{\eta}}|$ such that $(X_{\eta}, B|_{X_{\eta}} + H_{1,\eta}+H_{2,\eta})$ is dlt. Assuming this claim, we pick $H_i$ to be the closures of $H_{i,\eta}$ in $X$. Then $H_{i} \sim_{\Q} L_i + \sum a_i S_i$ for irreducible divisors $S_i$ of $S$ and some $a_i \in \Q$. Since $S_i$ are $f$-numerically trivial, we get that $H_1$ is $f$-ample, $H_2$ is $f$-anti-ample, and $H_1 + H_2$ is $f$-numerically trivial.

Thus, we are left to show the claim. In the mixed characteristic case it follows automatically from the standard Bertini results in characteristic zero. Hence we can assume that we are in the purely positive characteristic case. First, we replace $C$ by $C \,\backslash\, \{s\}$, and so we can assume that $S=0$ and $L_i$ are very ample. Let $V \subseteq X$ be the subset of points on which $(X,B)$ is strongly F-regular. By standard argument, $V$ is open. Now, by localising at codimension two points and applying \cite[Theorem 5.7]{satotakagi}, we get that $Z := X \, \backslash\, V$ is at most one dimensional. 

Let $H_{1}$ be a general member of $|L_1|$. By \cite[Corollary 6.10]{SZ13}, $(V,B|_V+H_1|_V)$ is purely F-regular (note that \cite{SZ13} uses a different name: divisorially F-regular), and hence plt with $H_1|_V$ normal. Since $H_1 \cap Z$ is zero-dimensional, we can replace $C$ by an open subset, so that $(X,B+H_1)$ is plt and $H_1$ is normal.

Let $H_2$ be a general member of $|L_2|$. Since $m_2 \gg m_1$, Serre's vanishing implies that $H^0(X, L_2) \to H^0(H_1, L_2|_{H_1})$ is surjective, and so $H_2|_{H_1}$ is a general member of $L_2|_{H_1}$. By the same argument as above, we have that $(X,B+H_1+H_2)$ is plt outside $H_1 \cap H_2$ up to replacing $C$ by an open subset. Write $K_{H_1} + B_{H_1} = (K_X+B+H_1)|_{H_1}$. By \cite[Theorem 5.9]{satotakagi} and inversion of adjunction for threefolds, we get that $(H_1, B_{H_1}+H_2|_{H_1})$ is plt. In particular, the inversion of adjunction for fourfolds (Corollary \ref{c-inversion-of-adjunction}) and the standard argument implies that the only log canonical centres of $(X,B+H_1+H_2)$ are the generic points of $H_1\cap H_2$. The generic points of $H_1\cap H_2$ are contained in the smooth locus $X^{\rm sm} \subseteq X$, and hence $(X,B+H_1+H_2)$ is simple normal crossing at them by the standard Bertini theorem applied to the smooth locus of $H_2$ (cf.\ \cite[Theorem 3]{JS12}).
\end{proof}

Let us point out that some of the assumptions on the existence of log resolutions of singularities may be weakened. 
\begin{conjecture}[Embedded resolution of singularities] \label{c-1} Let $X$ be a regular quasi-projective variety and $Z$ a closed subscheme, then there exists a projective birational morphism $\nu \colon X'\to X$ such that the schemes associated to ${\rm Ex}(\nu)$ and $\nu ^{-1}(Z)$ are divisors (in particular $\nu ^{-1}I_Z$ is locally free)  and $\nu ^{-1}Z\cup {\rm Ex}(\nu)$ has simple normal crossing support.
\end{conjecture}
From this conjecture we can deduce the existence of log resolutions for varieties birational to a regular variety.
\begin{proposition} \label{p-resolutions} Assume that Conjecture \ref{c-1} holds for a regular projective variety $X$ and all of its subschemes. Let $Y$ be a quasi-projective variety birational to $X$. Then for any scheme $W\subset Y$, there exists a proper birational morphism from a regular variety $\mu\colon  Y'\to Y$ such that schemes associated to ${\rm Ex }(\mu)$ and $\mu ^{-1}(W)$ are divisors and ${\rm Ex}(\mu) \cup \mu ^{-1}(W)$ has simple normal crossing support.
\end{proposition}
\begin{proof} Let $f \colon X \dasharrow Y$ be a birational map and let $Z$ be the indeterminacy locus of $f$. More precisely, let $V$ be the normalisation of the graph of $f$ with projections $p\colon V\to X$ and $q\colon V\to Y$. If $H$ is a very ample divisor on $Y$ and $H_1,\ldots , H_{d+1}\in |H|$ are general elements, where $d=\dim Y$, then $Z$ is cut out by $H'_i=p_*q^*H_i\in |H'|$, where $H'$ is the strict transform of $H$. 

 Let $\nu\colon X'\to X$ be the resolution of $Z$ given by Conjecture \ref{c-1}. Then $\nu ^*H'=F+M$ and $\nu ^* H'_i=F+M_i$ where 
 $F$ is a simple normal crossings divisor and the divisors $M_1,\ldots ,M_{d+1}$ give rise to a base point free linear series. Consider the corresponding morphism $f'\colon X'\to Y$.

We will now apply Conjecture \ref{c-1} to ${\rm Ex}(f')$ and $(f')^{-1}(W)$. Let $g\colon Y'\to X'$ and $\mu \colon Y'\to Y$ be the induced morphisms. Since ${\rm Ex}(\mu)={\rm Ex }(g)\cup g^{-1}({\rm Ex}(f'))$, then ${\rm Ex}(\mu)$ is a divisor, and so is $\mu ^{-1}(W)=g^{-1}((f')^{-1}(W))$. It is clear that their union has simple normal crossings.
\end{proof}

\section{Liftability to characteristic zero}

The goal of this section is to show Theorem \ref{t-mmp-liftability} and Theorem \ref{t-calabi-yau-liftability}. Throughout this section $k$ is a perfect field of characteristic $p>0$. 

Although we do not know that plt centres of four-dimensional pairs are normal in general, we can show this in the setting of lifting using deformation theory. This will be important in the later parts of this section.  
 \begin{lemma}\label{l-S3}
Let $\XX\to \Spec R$ be a projective morphism from a normal scheme to a complete DVR $(R,\fram)$ of mixed characteristic $(0,p>5)$ with perfect residue field. Suppose that the special fibre $X$ is three-dimensional, reduced, and $(\XX, X)$ is plt. Then $X$ is normal.

Moreover, if $(\XX, X)$ admits a log resolution of singularities, then the assumption that $(R,\fram)$ is complete may be dropped.
\end{lemma}
\begin{proof}
As for the last part: the completion $\hat{R} \to R$ is faithfully flat (see \cite[Tag 00MC]{stacks-project}), so we can replace $R$ by $\hat{R}$ and assume that $(R,\fram)$ is complete (see \cite[Tag 033G]{stacks-project}). The plt-ness of $(\XX,X)$ is preserved by the same argument as in \cite[Lemma 2.7]{BMPSTWW20} (here we use the existence of a resolution).

Let $\nu:X^\nu\to X$  be the normalisation. Since $(\XX, X)$ is plt, Lemma \ref{c-nuuh} shows that the special fiber $X=\XX_\mathfrak m$ is normal up to universal homeomorphism. By localising at points of $X$ of dimension one and applying \cite[Corollary 7.17]{BMPSTWW20} we get that that $X$ is normal except at possibly finitely many points; explicitly let $Z\subset X^\nu$ be a finite set of points such that $\nu |_{X^\nu\setminus Z}$ is an isomorphism onto its image. We then have that 
\[
K_{X^\nu}=(K_\XX+X)|_{X^\nu}
\]
is $\Q$-Cartier.  Moreover, as $(\XX, X)$ is plt, we get that $X^\nu$ is klt, and so Cohen-Macaulay (\cite[Corollary 1.3]{ABL20}). 

Let $U=X^\nu\setminus Z$. Since $X$ lifts to any order (to $X_i = \mathcal{X} \times_R R/\fram^{i+1}$) and $U \simeq \nu(U) \subseteq X$, we have that $U$ also lifts to any order. Explicitly, let $U_i$ be the induced compatible lifts of $U$ to $R/\fram^{i+1}$.

By \cite[Lemma A.23]{zdanowicz17}, $X^\nu$ also lifts to any order. For the convenience of the reader we explicate this result. Consider short exact sequences
\[
0\to \mathcal O _{U}\to \mathcal O _{U_i}\to\mathcal O _{U_{i-1}}\to 0
\]
of sheaves of rings on $|U|$. Define a sheaf $\OO_{X^{\nu}_i}$ of abelian groups on $|X^{\nu}|$ by the formula:
\[
\OO_{X^{\nu}_i} = i_*\OO_{U_i},
\]
where $i \colon U \to X^{\nu}$ is the inclusion of topological spaces. 
Pushing forward the above short exact sequence of abelian groups by $i$ and using normality of $X^{\nu}$ we get:
\begin{equation} \label{eq:lift-ses-flatness}
0\to \mathcal O _{X^\nu}\to \mathcal O _{X^{\nu}_i}\to\mathcal O _{X ^\nu _{i-1}}\to R^1i_*\OO_U=0.
\end{equation}
Here, $R^1i_*\OO_U$ is calculated in the category of abelian groups, but the answer is the same when it is calculated in the category of coherent sheaves (for example thanks to Čech resolution, cf.\ \cite[Tag 09V2]{stacks-project}), and so it is zero as $X^\nu$ is $S_3$ and $Z$ has codimension $\geq 3$. Finally, $\mathcal O _{X^{\nu}_i}$ inherits a structure of a ring from $U_i$, and so defines a scheme $X^{\nu}_i$ lifting  $X^{\nu}$ over $R/\fram^{i+1}$ (here, the flatness follows from the exactness of (\ref{eq:lift-ses-flatness})).  

Let $j=\nu \circ i: U\to X$. Then we have compatible homomorphisms
\[ \mathcal O _{X_i}\to j_* \mathcal O _{U_i}=\nu _* \mathcal O _{X ^\nu _i},
\]
and so compatible maps of schemes $X^{\nu}_i \to X_i$. Hence, we can apply \cite[Tag 09ZT and Tag 09ZW(8)]{stacks-project}, to get a finite map of schemes $\psi \colon \overline{\XX} \to \XX$ which reduces to $X^{\nu}_i \to X_i$ modulo $\fram^{i+1}$.

Moreover, $\psi \colon \overline{\XX} \to \XX$ is birational. Indeed, this is equivalent to showing that the coherent sheaf $ \psi_*\OO_{\overline{\XX}}/\OO_{\XX}$ is not supported on all of $\XX$, which can be checked after localising at a generic point $\eta_X$ of $X \subseteq \XX$. As $ \psi_*\OO_{\overline{\XX}}/\OO_{\XX}$ is zero on $\eta_X$, this yields a finite module $M$ over a Noetherian $R$-algebra $A = \OO_{\XX,\eta_X}$ such that $M/\fram M=0$ (equivalently, $\fram M = M$), and so there exists $f \in 1+\fram$ satisfying $fM=0$ by Nakayama's lemma (see \cite[Tag 00DV(1)]{stacks-project}). In particular, $\Ann(M) \neq 0$, and so $\Supp\, M \neq \Spec A$ as $A$ is integral.   


Since $\XX$ is normal, $\psi \colon \overline{\XX} \to \XX$ must thus be an isomorphism, and so $X\simeq X^{\nu}$ is normal concluding the proof of the lemma.
\qedhere


\end{proof}

\subsection{Liftability of minimal models to characteristic zero}

\begin{defn}
Let $(R, \fram)$ be a DVR of mixed characteristic and with residue field $k$. We say that a projective scheme $X$ over $k$ \emph{lifts over $R$} if there exists a flat projective morphism $\XX \to \Spec R$ with the central fibre $\XX_{\fram}$ isomorphic to $X$.

We say that a projective scheme $X$ defined over $k$ \emph{lifts to characteristic zero} if there exists a DVR $(R, \fram)$ of mixed characteristic and with residue field $k$ such that $X$ lifts over $R$.
\end{defn} 
Note that if $X$ is normal, then so is $\XX$. Indeed, $\XX$ is regular in  neighbourhoods of regular points of $X$ (cf.\ \cite[Tag 00NU]{stacks-project}),
and so it is $G_1$, and since $X$ is $S_2$ and Cartier, $\XX$ is $S_3$. 

In what follows we shall often write $S = \Spec R$.

\begin{lemma}\label{l-lblift}
Let $(R, \fram)$ be a complete DVR of mixed characteristic $(0,p)$ with residue field $k$. Let $X$ be a projective scheme over $k$ and which lifts to $\XX$ over $R$. Further, suppose that $H^2(X, \mathcal O _X)=0$.

Then any line bundle $L$ on $X$ lifts to a line bundle $\mathcal L$ on $\XX$ such that $\mathcal L|_X=L$ (in particular, $N^1(\XX/\Spec R) \to N^1(X)$ is surjective). If moreover $L$ is ample, then so is $\mathcal L$. 
\end{lemma}
\begin{proof} Let $X_n=\XX \times _S {\rm Spec }(R/\mathfrak m ^{n+1})$ and suppose that $L$ on $X$ lifts to a line bundle $L_n$ on $X_n$. By \cite[Theorem 6.4]{hartshorne_deformation}, the obstruction to lifting 
$L_n$ on $X_n$ to $L_{n+1}$ on $X_{n+1}$ lies in $H^2(X, \mathcal O _X)$ and hence vanishes by assumption. 
By \cite[Tag 08BE]{stacks-project}, there is a line bundle $\mathcal L$ on $\XX$ such that $\mathcal L|_{X_n}\cong L_n$ for every $n>0$.

The last assertion follows from the fact that ampleness may be checked over closed points (see \cite[Tag 0D3A]{stacks-project}).
\end{proof}
  
\begin{lemma}\label{l-cartier} Let $\XX$ be a normal scheme over a DVR $(R, \fram)$ of characteristic $(0,p)$ for $p\geq 0$. Suppose that the central fibre $X$ of $\XX$ is a normal $S_3$ scheme. 

Let $\mathcal D$ be a Weil divisor on $\XX$ whose support does not contain $X$ and such that $D=\mathcal D|_X$ is $\Q$-Cartier. Further, assume that there is a closed subset $Z\subset \XX$ satisfying ${\rm codim}(Z\cap X,X)\geq 3$ and for which  $D|_{\XX \setminus Z}$ is $\Q$-Cartier. Then  $\mathcal D$ is also $\Q$-Cartier. 
\end{lemma}
\noindent Note that since $X$ is normal, it is regular in codimension 1 and as observed above, so is $\XX$ (cf.\ \cite[Tag 00NU]{stacks-project}). Therefore $\mathcal{D}|_X$ is defined by taking the closure of the restriction
to the regular locus.
 \begin{proof}
  See \cite[12.1.8]{KM92} and  \cite[Proposition 3.1]{dFH11}. 
 \end{proof}
 \begin{corollary}\label{c-a} Let $\XX$ be a normal scheme over a DVR $(R, \fram)$ of characteristic $(0,p)$ for $p> 5$. Suppose that the central fibre $X$ is a three-dimensional terminal variety. Then
 \begin{enumerate}
     \item $\mathcal X$ is Cohen-Macaulay, $K_\XX$ is $\Q$-Cartier and if $X$ is $\Q$-factorial,  then so is $\mathcal X$,
     \item if $X$ is $\Q$-factorial and log resolutions of all log pairs with the underlying variety being birational to $\XX$ exist (and are constructed by a sequence of blow-ups along the non-snc locus), then $(\XX, X)$ is plt and $\mathcal X$ is terminal.  
 \end{enumerate}
 \end{corollary}
\noindent Note that the terminality of $\XX$ in (2) is not used later on, but we believe that it is of independent interest.
 \begin{proof}
   By \cite[Corollary 1.3]{ABL20}, $X$ is Cohen-Macaulay and as $X$ is a Cartier divisor, $\XX$ is also Cohen-Macaulay.
  Since $X$ is terminal, it is normal and  regular in codimension 2, thus so is $\XX$ (cf.\ \cite[Tag 00NU]{stacks-project}). 
  
  
  For the rest of (1), pick a divisor $\mathcal D$ whose support does not contain $X$ and such that $\mathcal D|_X$ is $\Q$-Cartier. We claim that then $\mathcal D$ is also $\Q$-Cartier. Replacing $\mathcal D$ by a multiple, we may assume that $\mathcal D|_X$ is Cartier. Notice that if $U=\XX \setminus \mathcal Z$ is the regular locus, then $\mathcal D$ is Cartier on $\XX \setminus \mathcal Z$ and ${\rm codim}( Z, X)\geq 3$ where $Z=\mathcal Z\cap X$. Hence the hypotheses of Lemma \ref{l-cartier} are satisfied, and so $\mathcal D$ is Cartier. 
  
  In particular, the above paragraph shows that $K_{\XX}$ is $\Q$-Cartier, as $K_X$ is $\Q$-Cartier and $K_\XX|_X=K_X$.
 Suppose now that $X$ is $\Q$-factorial and let $\mathcal D$ be a $\Q$-divisor on $\XX$. We need to show that $\mathcal D$ is $\Q$-Cartier. Since $X \subseteq \XX$ is Cartier and irreducible, we may assume that the support of $\mathcal D$ does not contain $X$. Since $X$ is $\Q$-factorial, we have that $\mathcal D|_X$ is $\Q$-Cartier, and so by the above paragraph, $\mathcal D$ is $\Q$-Cartier as well, concluding the proof of the $\Q$-factoriality of $\XX$.\\
  
  
  Now we move to (2) assuming the existence of appropriate resolutions. Note that as $X$ is terminal, it is klt, and so $(\XX,X)$ is plt by Corollary \ref{c-inversion-of-adjunction}. 
  
  Next, we construct a terminalisation of $\mathcal{X}$. Due to the limitations of our mmp results (Theorem \ref{t-mmp+} and Theorem \ref{t-mmp-families+}), we need to proceed in a non-standard way. More precisely, we pick a log resolution $\phi : \YY \to \XX$ of $(\mathcal{X},X)$ and run a $(K_{\YY} + Y)$-mmp over $\XX$ where $Y$ is the reduced special fibre (see Theorem \ref{t-mmp-families+}). The output of this mmp is a terminalisation of $\XX$, because:
  \begin{enumerate}
      \item All $\phi$-exceptional divisors on $\YY$ lying over $X$ are contracted as $(\XX, X)$ is plt. These divisors have log discrepancy with respect to $(\XX,0)$ greater than one, as $(\XX, X)$ is plt and $X$ is Cartier.
      \item A $\phi$-exceptional divisor on $\YY$ which is dominant over $\Spec R$ is contracted if and only if its log discrepancy with respect to $(\XX, 0)$ is greater than one (this can be checked over the generic point of $\Spec R$).
  \end{enumerate}
Hence, by replacing $(\YY,Y)$ by the output of this MMP, we may assume that $\YY \to \XX$ is a terminalisation and $(\YY,Y)$ is plt. Moreover, $Y = \phi^{-1}_*X$.

Now, assume that $X$ is terminal, but suppose by contradiction that $\XX$ is not terminal, or equivalently that $\phi \colon \YY \to \XX$ is not an isomorphism. Let $\mathcal{E}$ be a $\phi$-exceptional prime divisor on $\YY$. Since $Y$ is the reduced special fibre and $\Spec R$ admits only one closed point, we have that $E = 
\mathcal{E} \cap Y$ is non-empty.

By (1), $\mathcal{E}$ is dominant over $\Spec R$. In particular, the irreducible subscheme $\phi(\mathcal{E})$ is also dominant over $\Spec R$ and of codimension at least two in $\XX$. Hence, $\phi|_Y(E) \subseteq X$ is of codimension at least two in $X$, and so $E \subseteq Y$ is exceptional over $X$. This contradicts the fact that $X$ is terminal as 
\begin{align*}
K_Y &= K_{\YY}|_Y = (\phi^*K_{\XX} + G)|_Y = (\phi|_Y)^{-1}K_X + G|_Y \\
K_Y &= (\phi|_Y)^*K_X + F,
\end{align*}
where $G \leq 0$ as $\YY \to \XX$ is a terminalisation, and $F \geq 0$ with $\Supp\, F = \mathrm{Exc}(\phi|_Y)$ (in particular, it contains the exceptional divisor $E$ with positive coefficient) as $X$ is terminal.

 \end{proof}

 \begin{proof}[{Proof of Theorem \ref{t-mmp-liftability}}] By \autoref{c-a},  $\mathcal X$ is Cohen Macaulay,  $\Q$-factorial, terminal and  $(\XX,X)$ is plt. 
 Let $\eta$ be the generic point of $S=\Spec R$, then we claim that $K_{X_\eta}$ is pseudo-effective.
 
 To see this let  $\mathcal{A}$ be an ample Cartier divisor on $\mathcal{X}$ and let $t\in \mathbb Q$ be the pseudo-effective threshold (or $0$, if $K_{\XX}$ is already pseudo-effective) 
 so that $0\leq \kappa (K_{\XX_\eta}+t\mathcal{A}_\eta)$ (the fact that $t$ is rational follows by running a characteristic zero $K_{X_{\eta}}$-mmp with scaling of $\mathcal{A}_{\eta}$ and the cone theorem applied to a resulting Mori fibre space if $t>0$).
 We wish to show that $t=0$. 
  We may assume that $(\XX,X+t\mathcal A)$ is plt and $(X,tA)$ is terminal where $A=\mathcal A |_X$.
 By Theorem \ref{t-mmp-families+}, we may run the $K_{\XX}+X+t\mathcal{A}$ minimal model program with scaling over $S$
  which terminates with a minimal model $\phi \colon \XX\dasharrow \XX'$. Let $X'$ be the special fibre of $\XX'$ over $\Spec R$; equivalently $X' = \phi_*X$. By definition, $X'$ is Cartier and $K_{\XX'}+X'+t\mathcal A'$ is nef over $S$. 
  As $(\XX ',X'+t\mathcal A')$ is plt (and so is $(\XX',X')$ in view of $\XX'$ being $\Q$-factorial),  Lemma \ref{l-S3} shows that $X'$ is normal.
  Note that as $X$ is terminal, 
 \[
 \phi|_X \colon X\dasharrow X'
 \]
 does not extract any divisors. Indeed, let $\mathcal A'=\phi_* \mathcal A$ and  $A'=\mathcal A'|_{X'}$, then every divisor on $X'$, say $E \subseteq X'$, has log discrepancy $\leq 1$ with respect to $(X',tA')$, and so log discrepancy $\leq 1$ with respect to $(X,tA)$  (running a $(K_{\XX}+X+t\mathcal A)$-mmp can only increase the discrepancies of $(\XX,X+t\mathcal A)$, and hence cannot decrease those of $(X,tA)$ by adjunction). Thus it is impossible for $E$ to be extracted, because then $E$ is not a divisor on $X$, and so by definition of terminality has log discrepancy $>1$.

  Note that $K_{X'}+tA'$ and $K_{\XX'_\eta }+t\mathcal A'_\eta$ are nef and if $t>0$, then $K_{X' }+t A'$ is big. It then follows that
  \[ (K_{\XX'_\eta }+t\mathcal A'_\eta)^3= (K_{X'}+tA')^3>0,\]
 so that $K_{\XX'_\eta }+t\mathcal A'_\eta$ is also big, a contradiction to $t$ being positive.
 Therefore, we may assume that $t= 0$ and hence $K_{\XX'}+X'$ is nef over $S$.
  Since $ K_{X'}=(K_{\XX'}+X')|_{X'}$ is $\Q$-Cartier and
 $(\XX ',X')$ is plt, we get that $X'$ is klt and $K_{X'}$ is nef. It follows that $X\dasharrow X'$ is a (possibly non-$\Q$-factorial) minimal model  concluding the proof of (1).\\

 We move on to the proof of (2a). Suppose now that $N^1(\XX/S)\to N^1(X)$ is surjective. In particular, $\overline{\mathrm{NE}}(\mathcal{X}/S)=\overline{\mathrm{NE}}(X)$; indeed this is true because only curves which are projective over $S$ (and hence lie on the special fibre $X$) are used in the definition of $\overline{\mathrm{NE}}$ (cf.\ \cite{tanaka16_excellent} or \cite{BMPSTWW20}) and numerical equivalence relations on $X$ and $\XX$ are the same in view of $N^1(\XX/ S)\to N^1(X)$ being surjective.  
 
 Our goal is to show that any sequence of $K_X$-mmp steps 
 \[
 X \xdashrightarrow{h} X_1 \dashrightarrow X_2 \dashrightarrow \ldots \dashrightarrow X_k
 \]
 (where $X_k$ is a minimal model of $X$) lifts to a sequence of $K_{\XX}$-mmp steps 
 \[
 \XX \xdashrightarrow{\phi} \XX_1 \dashrightarrow \XX_2 \dashrightarrow \ldots \dashrightarrow \XX_k.
 \]
 Consider the first step $h \colon X \dashrightarrow X_1$. Let $A$ be an ample $\Q$-divisor on $X$ such that $(K_X+A)^{\perp} \subseteq \overline{\mathrm{NE}}(X)$ is equal to the extremal ray $\Sigma$ defining $h$. Since $N^1(\XX/S)\to N^1(X)$ is surjective, there exists an ample $\Q$-divisor $\mathcal{A}$ on $\XX$ such that $\mathcal{A}|_X \equiv A$. Then $(K_{\XX}+X+\mathcal{A})^{\perp} = \Sigma \subseteq \overline{\mathrm{NE}}(\mathcal{X}/S)$. 
 
 We set $\phi \colon \mathcal X \dashrightarrow \mathcal X_1$ to be the step of a $(K_{\XX}+X)$-mmp (equivalently, $K_{\XX}$-mmp) corresponding to $\Sigma$. In particular, $(\XX_1, (\XX_1)_{\fram})$ is plt, where $(\XX_1)_{\fram}$ is the special fibre, hence $(\XX_1)_{\fram}$ is normal by Lemma \ref{l-S3}. As explained before, the terminality of $X$ ensures that $\phi|_X \colon X \dashrightarrow (\XX_1)_{\fram}$ does not extract any divisors.
 
 We claim that $\phi|_X \colon X \dashrightarrow (\XX_1)_{\fram}$ and $h \colon X \dashrightarrow X_1$ coincide.  To this end, let $f \colon X \to Z$ be the contraction of $\Sigma \subseteq \overline{\mathrm{NE}}(X)$ induced by $K_{X}+A$ and let $\psi \colon \XX \to \ZZ$ be the contraction of $\Sigma \subseteq \overline{\mathrm{NE}}(\mathcal{X}/S)$ induced by $K_{\XX}+X+\mathcal{A}$. Note that $\psi|_X \colon X \to \ZZ_{\fram}$ contracts exactly $\Sigma$ and no other numerical class of a curve, hence we have a factorisation
 \begin{equation} \label{eq:lifting-factorisation}
 \psi|_X \colon X \xrightarrow{f} Z \to \ZZ_{\fram}
 \end{equation}
 and $Z$ is the normalisation of $\ZZ_{\fram}$.
 
 Before proceeding with the proof of the claim, we will show that $\ZZ_{\fram} \simeq Z$ and $\psi|_X = f$.  To this end, we may assume that $(X,A)$ is klt (cf.\ \cite[Lemma 2.29]{BMPSTWW20} or \cite[Lemma 9.2]{birkar13}; this uses the existence of a resolution with an ample exceptional divisor). Hence, $(Z, f_*A)$ is also klt, and so both $X$ and $Z$ have rational singularities by \cite[Corollary 1.3]{ABL20}. Therefore, $R^1f_*\OO_X=0$, thus $R^1(\psi|_X)_*\OO_{\XX}=0$ as well, which in turn implies that $R^1\psi_*\OO_{\XX}=0$ by upper-semicontinuity. Hence, by the projection formula
     \[ R^1\psi_*\OO_{\XX}(m(K_{\XX}+\mathcal{A}))=0
     \]
     for $m$ divisible enough, and so
     \[
     \psi_*\OO_{\XX}(m(K_{\XX}+\mathcal A)) \to (\psi|_X)_*\OO_X(m(K_X+A))
     \]
     is surjective. In particular, $\ZZ_{\fram} \simeq Z$ and $\psi|_X = f$. 
     
Now we consider the following cases:
\begin{enumerate}
    \item $\phi \colon \mathcal{X} \dasharrow \mathcal{X}_1$ is a divisorial contraction (that is, $\psi=\phi$): then $\psi|_X=f$ is also divisorial by semicontinuity of fibre dimension; in particular, the first step $h \colon X \dashrightarrow X_1$ of the $K_X$-mmp is the contraction $f \colon X \to Z$, and hence it lifts to $\phi=\psi \colon \XX \to \ZZ$. 
    \item $\phi \colon \mathcal{X} \dasharrow \mathcal{X}_1$ is a flip over the flipping contraction $\psi \colon \XX \to \ZZ$ (in particular, $\ZZ$ is not $\Q$-Gorenstein) and
    \begin{enumerate}
        \item $f \colon X \to Z$ is divisorial: then $Z$ is $\Q$-Gorenstein,  Cohen-Macaulay (as it has rational singularities), and $\mathrm{codim}\,  \mathrm{Sing}(Z) \geq 3$ (as $X$ is terminal). Since $Z \simeq \ZZ_{\fram}$, we get that $\ZZ$ is $\Q$-Gorenstein by Lemma \ref{l-cartier} (cf.\ the proof of Corollary  \ref{c-a}) which is impossible. This shows that $f$ can never be divisorial when $\phi$ is a flip.
        \item $f \colon X \to Z$ is small: then, $(\XX_1)_{\fram}$ is normal (cf. Lemma \ref{l-S3}), and  since $\phi|_X \colon X \dashrightarrow (\XX_1)_{\fram}$ does not extract any divisors, the induced map $(\XX_1)_{\fram} \to Z$ is also small. As $(\XX_1)_{\fram}$ is $\Q$-Gorenstein and $K_{(\XX_1)_{\fram}} = K_{\XX_1}|_{(\XX_1)_{\fram}}$ is ample over $Z$, we have that $(\XX_1)_{\fram}$ is the canonical model of $X$ over $Z$. Hence it is a flip, $(\XX_1)_{\fram} \simeq X_1$ and $\phi|_X=h$.
    \end{enumerate}
\end{enumerate}

Having lifted $h \colon X \dashrightarrow X_1$, we would like to replace $X$ and $\XX$ by $X_1$ and $\XX_1$, respectively and repeat the above procedure. To be able to do this, we deduce the surjectivity of $N^1(\mathcal{X}_1/S) \to N^1(X_1)$ from that of $N^1(\mathcal{X}/S) \to N^1(X)$. Pick a line bundle $L$ on $X_1$. We need to show that $L^m \equiv \mathcal{L}|_{X_1}$ for some line bundle $\mathcal L$ on $\XX_1$ and $m \in \mathbb{N}$. This is a standard consequence of the base point free theorem as explained by the following case by case analysis:
\begin{enumerate}
    \item $\phi=\psi \colon \mathcal{X} \to \mathcal{Z}$ is a divisorial contraction (in particular, $Z=X_1$): then up to replacing $L$ by a multiple, we may pick $\mathcal{L}'$ on $\XX$ such that $f^*L \equiv \mathcal{L}'|_X  $. Since $\mathcal{L}'$ is numerically trivial over $\ZZ$, it descends, up to some multiple, to a line bundle $\mathcal{L}$ on $\ZZ$ by the base point free theorem (cf.\  Proposition \ref{p-bpf-families}), and $\mathcal{L}|_Z \equiv L$ by construction.
    \item $\phi \colon \XX \dashrightarrow \XX_1$ is small: then since $K_{\XX_1}|_{X_1} = K_{X_1}$ is ample over $Z$ and $\rho(X_1/Z)=1$, we can replace $L$ by $L+aK_{X_1}$ for some $a \in \Q$, and assume that $L\equiv_Z 0$. Then by the base point free theorem (cf.\  Proposition \ref{p-bpf-families}), up to some multiple, $L$ descends to a line bundle $L_Z$ on $Z = \ZZ_{\fram}$. By the same argument as in (1), up to some multiple, there exists $\mathcal{L}_{\mathcal Z}$ on $\ZZ$ such that $\mathcal{L}_{\mathcal Z}|_Z \equiv L_Z$. Then $(\psi^+)^*\mathcal{L}_{\mathcal Z}|_{X_1} \equiv L$, where $\psi^+\colon \XX_1 \to \ZZ$ is the flipped contraction.
\end{enumerate} 
Thus we can repeat the above procedure until we get a lift $\mathcal{X}_k$ of $X_k$. Since $K_{X_k}$ is nef, so is $K_{\mathcal{X}_k}$. Hence $\mathcal{X}_k$ is a minimal model of $\XX$.\\
 
 Next we tackle (2b). By (2a), we can replace $X$ by the output of a $K_X$-mmp, and so assume that $X$ is minimal. Let $h \colon X \dashrightarrow Y$ be a small birational map to another $\Q$-factorial minimal model $Y$. It is well known (cf.\  \cite{Kaw08}) that $X\dasharrow Y$ is a small birational map given by a finite sequence of flops. To see this, 
 pick $A_Y$ an ample Cartier divisor on $Y$ and let $A$ be its strict transform on $X$. Since $X\dasharrow Y$ is small, $A$ is $\Q$-movable and big. Pick $0\leq D\sim _\Q A$. We may
 assume that $(X,\epsilon D)$ is terminal. 
 We claim that for $0<\epsilon \ll 1$ there is a finite sequence of $(K_X+\epsilon D)$-flips which are $K_X$-flops inducing a $(K_X+\epsilon D)$-minimal model $\phi: X\dasharrow X'$.
 We proceed as follows. Suppose that $rK_X$ is Cartier, now
 if $K_X\cdot C\ne 0$ for some curve $C$, then we have $K_X\cdot C\geq 1/r$. If $K_X+tD$ is nef for some $0<t< \epsilon/(6r+1)$, we may replace $\epsilon $ by $t$ and we are done.  If on the other hand $K_X+tD$ is not nef for any $0<t<\epsilon /(6r+1)$, then let $\Sigma$ be a negative extremal ray generated by a curve $C$. Note that $\Sigma$ is also a $(K_X+\epsilon D)$-negative extremal ray and so we may assume that $(K_X+\epsilon D)\cdot C\geq -6$ (cf. \cite[Theorem 1.1]{DW19}).
 But then \[(K_X+t D)\cdot C=(1-\frac t\epsilon)K_X\cdot C +\frac t\epsilon (K_X+\epsilon D)\cdot C\geq (1-\frac t\epsilon)\cdot \frac 1 r -\frac {6t}\epsilon>0.\]
  This is impossible and so $K_X\cdot C=0$. It follows that $\Sigma$ is a $(K_X+\epsilon D)$-negative extremal ray which is $K_X$ trivial. Let $X\dasharrow X^+$ be the corresponding flop, then $K_{X^+}$ is nef and $rK_{X^+}$ is Cartier (cf.\ \cite[Corollary 1.5]{ABL20}). We may replace $X$ by $X^+$ and repeat the procedure which terminates after finitely many steps by Proposition \ref{proposition:AHK} (note in fact that each flip is $D$-negative and hence the flipping locus is always contained in the support of $D$).
 
 Let $\phi: X\dasharrow X'$ be the corresponding minimal model. Since $D$ is movable, $\phi$ is an isomorphism in codimension 1.  Since $K_Y+\epsilon A_Y$ is ample, the (unique) canonical model of $K_X+\epsilon D$ is isomorphic to $Y$ and the induced morphism $X'\to Y$ is a small birational morphism of $\Q$-factorial varieties and hence an isomorphism.
 Therefore we may assume that $X\dasharrow Y$ is a $K_X$-flop and we must show that $Y$ lifts to $\mathcal Y$ over $R$ and $N^1(\mathcal Y/{\rm Spec}(R))\to N^1(Y)$ is surjective.

 Let $f\colon X\to Z$ be the flopping contraction and $H=f^*H_Z$ the pullback of an ample divisor $H_Z$ on $Z$. We also set $\Sigma\in \overline {\rm NE}(X)$ to be the corresponding $(K_X+A)$-negative extremal ray.
 
 Since $N^1(\XX/S) \to N^1(X)$ is surjective, there exist Cartier divisors $\mathcal A$ and $\mathcal H$ on $\XX$ such that $\mathcal A |_X = A$ and $\mathcal H |_X = H$. 
 Since $H$ is nef and big, it is easy to see that $\mathcal H _\eta $ is nef and big as well, and hence $\mathcal H$ is big over $S$.
 Replacing $\mathcal H$ by a multiple, we may assume that $\mathcal H+ \mathcal A$ is big.
 Pick $0\leq \mathcal D\sim_{\mathbb{Q}} \mathcal H+ \mathcal A$ an effective $\Q$-divisor whose support does not contain $X$. For $0<\epsilon \ll 1$, we may assume that $(\XX,X+\epsilon \mathcal D)$ is plt and $(X,\epsilon D)$ is terminal where $D=\mathcal D |_X$. Arguing as in the proof of (2a), we have a $(K_\XX+\epsilon \mathcal D)$-flip $\XX \dasharrow \XX ^+$ corresponding to the extremal ray $\Sigma \in \overline {\rm NE}(\XX /S)$
 and $\XX ^+_{\mathfrak m}=Y$. The surjectivity of $N^1(\mathcal Y/{\rm Spec}(R))\to N^1(Y)$ also follows as in the proof of (2a). This concludes the proof of (2b).

Finally, we tackle (3). By (1), we can replace $X$ by a (possibly non-$\Q$-factorial) minimal model and assume that $K_X$ is nef and big. Note however that the lift $\XX$ of $X$ constructed in (1) is $\Q$-factorial. Now $K_{\XX} +X \sim_{\Q} K_{\XX}$ is also nef and big, and so it is semiample by Proposition \ref{p-bpf-families2}. Let $\psi \colon \mathcal{X} \to \mathcal{X}^c$ be the induced birational morphism. Since $(\mathcal{X}^c, (\mathcal{X}^c)_{\fram})$ is plt, Lemma \ref{l-S3} shows that  $(\mathcal{X}^c)_{\fram}$ is normal. In particular, $\psi|_X = f$, where $f \colon X \to X^c$ is the map induced by $K_{X}$. Therefore, $\mathcal{X}^c$ is the lift of $X^c$, which concludes the proof of (3).

 \end{proof}
 \begin{proof}[{Proof of Theorem \ref{t-calabi-yau-liftability}}]
By Lemma \ref{l-lblift}, $N^1(\XX/S) \to N^1(X)$ is surjective. Hence we can conclude by Theorem \ref{t-mmp-liftability}(2b) as $Y$ is a minimal model of $X$.
\end{proof}

\section*{Acknowledgements}{}
We would like to thank Bhargav Bhatt, J\'anos Koll\'ar, Karl Schwede, Zsolt Patakfalvi, Joe Waldron, and Chenyang Xu for comments and helpful suggestions.

The first author was partially supported by NSF research grants no: DMS-1300750, DMS-1952522, DMS-1801851 and by a grant
from the Simons Foundation; Award Number: 256202. He would also like
to thank the Mathematics Department and the Research Institute for Mathematical Sciences,
located Kyoto University and the Mathematical Sciences Research Institute in Berkeley were some of this research was conducted. The second author was supported by the National Science Foundation under Grant No.\ DMS-1638352 at the Institute for Advanced Study in Princeton, and by the National Science Foundation under Grant No.\ DMS-1440140 while the author was in residence at the Mathematical Sciences Research Institute in Berkeley, California, during the Spring 2019 semester. He is currently supported by the NSF research grant no: DMS-2101897.

\bibliographystyle{amsalpha}
\bibliography{final} 

\end{document}